\documentclass[10pt]{article}

\usepackage{amsmath, amssymb, amsthm}
\usepackage{fullpage}
\usepackage{graphics}
\usepackage{scalerel}
\usepackage{tikz}
\usetikzlibrary{decorations.pathreplacing,angles,quotes}
\usetikzlibrary{arrows.meta}
\usepackage{environ}
\usepackage{framed}
\usepackage{url}
\usepackage{algorithm}
\usepackage[noend]{algpseudocode}
\usepackage[labelfont=bf]{caption}
\usepackage{cite}
\usepackage{framed}
\usepackage[framemethod=tikz]{mdframed}
\usepackage{appendix}
\usepackage{graphicx}
\usepackage[textsize=tiny]{todonotes}
\usepackage{tcolorbox}
\usepackage{relsize}
\usepackage{todonotes}
\usepackage{bm}

\newcommand\remove[1]{}

\theoremstyle{plain}
\newtheorem{lemma}{Lemma}[section]
\newtheorem{corollary}[lemma]{Corollary}
\newtheorem{theorem}[lemma]{Theorem}
\newtheorem{conjecture}[lemma]{Conjecture}

\theoremstyle{definition}
\newtheorem{definition}[lemma]{Definition}

\newtheorem{example}[lemma]{Example}

\newcommand\C{\mathbb{C}}

\renewcommand{\sf}[1]{\textup{\textbf{#1}}}

\newcommand{\wt}{\widetilde}
\newcommand{\lc}{\left\lceil}
\newcommand{\rc}{\right\rceil}
\newcommand{\lf}{\left\lfloor}
\newcommand{\rf}{\right\rfloor}

\newif\ifrandom
\randomtrue

\author{Meghal Gupta}

\begin{document}
%\pagenumbering{gobble}

\title{A formula for $F$-Polynomials in terms of $C$-Vectors and Stabilization of $F$-Polynomials}

\clearpage\maketitle

\thispagestyle{empty}

\begin{abstract}
Given a quiver associated to a cluster algebra and a sequence of vertices, iterative mutation leads to $F$-Polynomials which appear in numerous places in the cluster algebraic literature. The coefficients of the monomials in these $F$-Polynomials are difficult to understand and have been an area of study for many years. In this paper, we present a general closed-form formula for these coefficients in terms of elementary manipulations with $C$-matrices. We then demonstrate the effectiveness of the formula by using it to derive simple explicit formulas for $F$-Polynomials of specific classes of quivers and mutation sequences. Work has been done to do these cases in ad-hoc combinatorial ways, but our formula recovers and improves known formulas with a general method. 

Secondly, we investigate convergence of $F$-polynomials. In themselves, they do not converge, but by changing bases using the $C$-matrix, they conjecturally do. Since our formula relates $C$-matrix entries to coefficients, we are able to apply it to make considerable progress on the conjecture. Specifically, we show stability for green mutation sequences. Finally, we look at exact formulas for these stable deformed $F$-polynomials in instances where they illustrate properties that the $F$-polynomials themselves do not.
\end{abstract}

\section{Introduction}
Cluster algebras, introduced by Fomin and Zelevinsky in 2002 \cite{FZ1}, are commutative algebras with a distinguished set of generators, known as cluster variables. Their original motivation came from an algebraic framework for total positivity and canonical bases in Lie Theory. In recent years, connections have been made to other areas, such as physics, quiver representations, Calabi-Yau categories, and Poisson geometry.

The cluster variables are grouped into sets of constant cardinality $v$, the \textbf{clusters}, and the integer $v$ is called
the rank of the cluster algebra. Start with an initial cluster $\sf{x}$ (together with a skew symmetrizable integer $v\times v$ matrix $B$ and a coefficient vector $\sf{p} = \left<(p_1^+,p_1^-),(p_2^+,p_2^-),\ldots,(p_v^+,p_v^-)\right>$ whose entries are elements of a torsion-free abelian group. We denote the entry in the $k$th row and $i$th column of $B$ as $B[k,i]$. The set of cluster variables is obtained by repeated application of so called \textbf{mutations}. To perform a mutation at $k$, replace $x_k$ in the cluster with $$x'_k=\frac{1}{x_k}\left(p_k^{+}\prod_{B[k,i]>0}x_i^{B[k,i]}+p_k^{-}\prod_{B[k,i]<0}x_i^{-B[k,i]}\right)$$ This creates a new cluster variable. By a result of Fomin and Zelevinsky \cite{FZ1} known as the Laurent phenomenon, this new cluster variable will always be a Laurent polynomial. There is also a corresponding change to the matrix $B$ and vector $p$, which will be defined properly in the next section. When one sets each initial $x_i$ to 1 and initial $p$ vector as $\left<y_1,\ldots,y_v\right>$, each cluster variable is a polynomial with indeterminates $\left<y_1,\ldots,y_v\right>$. These polynomials are called $F$-polynomials. It turns out that the coefficients of monomials in $F$-polynomials are the same as those we would have gotten by setting each initial $y_i$ to 1 and leaving the initial $x$ vector as indeterminates $\left<x_1,\ldots,x_v\right>$. As such, results about coefficients of $F$-polynomials also apply to cluster variables.

$F$-Polynomials have been an interesting area of study since the inception of the field of cluster algebras. Fomin and Zelevinsky conjectured in \cite{FZ1} that every $F$-polynomial has all positive coefficients. This conjecture was later proven in the skew-symmetric case in \cite{LS2} and then the entire skew-symmetrizable case in \cite{GHKK}. Nonetheless, we are far from completely understanding them. For one, not very many explicit formulas exist for the coefficients of $F$-polynomials. Since they are defined through recurrence relations, it is not clear that we can actually get a closed form formulation instead. The main result of our paper is to derive a closed form based only on the $C$-matrix entries, which are related to the $p$-vectors defined above. Some of the most general previous results regarding $F$-polynomial explicit forms are the following:
\begin{itemize}
\item A formula in terms of the Euler-Poincare characteristic of quiver Grassmannians obtained in \cite{DWZ}.
\item A combinatorial interpretation for cluster algebras from surfaces given in \cite{MSW}.
\item A formula for cluster variables corresponding to string modules as a product of 2 $\times$ 2 matrices obtained in \cite{ADSS}.
\item A formula using broken lines and reflections in \cite{GHKK}.
\end{itemize}
The last result is the only one that is completely general; the others are restrictive in the cluster algebras they account for. The first is mostly general, restricting only in that it only applies to skew-symmetric cluster algebras. However, the first and last have limited computational power, as the coefficients involve objects that are hard to understand and compute.

As said earlier, the main result of this paper is to present a formula for the $F$-polynomial that is completely general and completely explicit. This is in itself interesting, because it shows that the recurrence relation for $F$-polynomials can be solved out. Moreover, we show that our formula is useful. We demonstrate that this formula has computational power, meaning its terms are easy to understand, unlike the formula using the Euler-Poincare characteristic of quiver Grassmannians. We demonstrate this by using it to easily obtain algebraic formulas for $F$-polynomials of specific classes of quivers, that are as simple or simpler than previously known. Though our formula is general, our examples and applications of it restrict to the skew-symmetric case, because that is where patterns are most evident. Furthermore, the formula has an application to geometry in this case, where the cluster algebra can be interpreted with quivers. The formula for the $F$-polynomial in terms of the Euler-Poincare characteristic of quiver Grassmannians mentioned earlier in \cite{DWZ} says the following: 
$$F_n = \sum_e \chi(Gr_e(\phi_n))\prod_{i=1}^v y_i^{e_i}$$

Here, $e$ is a vector in $\mathbb{N}^v$ and $\phi_n$ is a representation of the original quiver defined by certain mutation rules. As such, our expression for the coefficient of a given monomial in an $F$-Polynomial also gives a formula for the corresponding Euler-Poincare characteristic of quiver Grassmannians associated to dimension vectors parameterized by $e$. This is interesting, because these are often very difficult to understand and compute. Finally, it is well known that $F$-polynomial coefficients for specific quivers represent various combinatorial objects; for example pyramid partitions \cite{EF}, T-paths in surfaces with marked points \cite{ST}, and snake graphs \cite{MS}. Therefore, our $F$-polynomial formula also provides a way to count these combinatorial objects.

Secondly, we use our formula to make progress on a conjecture made by Eager and Franco in Section 9.5 of \cite{EF}. They explore convergence of $F$-polynomials in the skew symmetric case. In themselves, $F$-polynomials do not converge, but when one changes basis; i.e. acts with an appropriate linear operator on the exponent vector of each monomial, they appear to converge. This phenomenon and precise linear operator was conjectured in \cite{EF} and formalized by Grace Zhang in \cite{Z}. This bears a striking resemblance to our formula for the coefficients of $F$-Polynomials, and we use our results to prove Eager and Franco's conjecture for a certain class of quivers known as green quivers. This includes the quiver they originally conjectured stabilization for, the $dP1$ quiver, known in combinatorics for giving rise to the Somos-4 sequence. This also extends results of \cite{Z}, which prove stabilization for some specific green quivers.

\section{Definitions}
We start with the general definitions, and later state the simplified definitions for the skew-symmetric case, since that is the case we study most closely. The way we define $F$-polynomials is purely combinatorial; we do not use the cluster algebra notation directly. In both cases, we would start with a matrix B and mutation sequence $v_1,\ldots,v_n$ to obtain the $F$-polynomial $F_n$. In our definitions, the matrix $B$ is essentially replaced by the notion of a generalized quiver defined below in Section \ref{generaldefs}. An interested reader should look at \cite{FZ1} and \cite{FZ4} to understand why the cluster algebra definition of an $F$-polynomial is the same.

Furthermore, the reader interested in only the skew-symmetric case can ignore the concept of ``attached'' edges, and instead look at Section \ref{simpledefs} for the quiver and mutation definitions. From there, the reader can come back to this section, and ignore any occurrence of the word ``attached'' and assume that all edges are attached to both incident vertices. In this case, we also have the matrix $A$ is the same as $E$ and $C$ the same as $D$.

\subsection{General Definitions} \label{generaldefs}
\begin{definition}[Generalized Quiver]
A \textbf{generalized quiver} $Q$ is a directed graph on vertices $1,\ldots,|Q|$, where every edge is ``attached'' to one of its incident vertices, along with a Laurent polynomial label at every vertex. The number of vertices $|Q|$ will be denoted $v$. However the edges must satisfy certain properties:
\begin{itemize}
\item There are no 2-cycles involving edges both attached to the same vertex.
\item When written in a matrix $B$ with entry $B[i,j]$ being the number of edges attached to $j$ going from $j \to i$ (negative if edges go in), $B$ is skew-symmetrizable.
\end{itemize}
\end{definition}

Intuitively, a generalized quiver is a generalized version of a quiver, reflecting the fact that the ``adjacency matrix'' $B$ need not be skew-symmetric.

\begin{definition}[Generalized Quiver Mutation]
An \textbf{generalized quiver mutation} at vertex $i$ is defined by changing both the quiver around vertex $i$ and the Laurent polynomial at vertex $i$. The quiver changes are defined as follows.
\begin{itemize}
\item For every path $j \to i \to k$ with the first edge attached to $i$ and the second to $k$, draw a new edge from $j$ to $k$ attached to $k$.
\item For every path $j \to i \to k$ the first edge attached to $j$ and the second to $i$, draw a new edge from $j$ to $k$ attached to $k$. Reverse the direction of every edge pointing with $i$ as a vertex.
\item Delete all 2-cycles involving edges both attached to the same vertex.
\end{itemize}
The Laurent polynomial is replaced as follows. Let $V_k$ denote the Laurent polynomial currently on vertex $k$. Replace $V_i$ with the following function $P$. $$P=\frac{\prod_{j \in S_1}V_j + \prod_{j \in S_2}V_j}{V_i}$$ where $S_1$ is the set of vertices with multiplicity pointing into $i$ attached to $i$ and $S_2$ is the set of vertices with multiplicity of edges pointing out of $i$ attached to $i$.
\end{definition}

Notice that the vertex $i$ may no longer have a Laurent polynomial label. However, the following theorem by Fomin and Zelevinsky rectifies that.

\begin{theorem}[Laurent Phenomenon \cite{FZ1}]
Take a generalized quiver $Q$ with starting polynomials $x_1,\ldots,x_v$ at each corresponding vertex. For any mutation sequence, the function $P$ as defined above is actually a Laurent polynomial.
\end{theorem}

\begin{definition}[Framed Generalized Quiver]
The \textbf{framed generalized quiver} of a generalized quiver $Q$ with labels (polynomials) $1$ at every vertex is the quiver formed by adding a vertex $i'$ for every vertex $i$ with an edge pointing from $i'$ to $i$ attached to $i$. The new vertices $i'$ have the label $y_i$. A \textbf{base vertex} of the framed generalized quiver is a vertex that is also a vertex of the original quiver. 
\end{definition}

\begin{definition}[$F$-Polynomial]
The $n$th \textbf{$\bm{F}$-Polynomial} of a framed generalized quiver $Q$ accompanied with a mutation sequence $v_1,\ldots,v_n$ of vertices is the polynomial at vertex $v_n$ after performing mutations at $v_1,\ldots,v_n$ sequentially.
\end{definition}

\begin{definition}[$V_{n}(i)$ or $V_i$]
Given a quiver $Q$ and a sequence of mutations $v_1,\ldots,v_n$, $\bm{V_n(i)}$ is the $F$-polynomial on vertex $i$ after the $n$ mutations. We will usually denote this $F$-polynomial more simply as $V_i$ if we are otherwise clear about the termination point $n$ for the sequence of mutations.
\end{definition}

\begin{definition}[$C$-matrix]
The \textbf{$\bm{C}$-matrix} of the framed generalized quiver $Q$ and mutation sequence $v_1,\ldots,v_n$ with $v$ base vertices is the $v \times v$ matrix with entry $(i,j)$ as the number of edges from frozen vertex $i'$ to base vertex $j$ attached to $j$. It is denoted $C_n$.
\end{definition}

\begin{definition}[$A_i$]
Given a framed generalized quiver $Q$ and a sequence of mutations $v_1,\ldots,v_i$, let $\bm{A_i^g}$ be the identity matrix with entry $(v_i,v_i)$ replaced with $-1$ and for every arrow pointing from $v_i$ to another base quiver vertex $u$, add 1 to the entry $(v_i,u)$. $\bm{A_i^r}$ is defined analogously with arrows pointing into $v_i$. We will just refer to the relevant $A_i^g$ or $A_i^r$ by $A_i$, where the relevant one is what corresponds to the color of $r_i$ as will be defined soon. The other one will be $A_i^*$.
\end{definition}

\begin{lemma}[$C$-Matrix Facts, Theorem 5.6 \cite{K}]
Given a framed quiver $Q$ and a sequence of mutations $v_1,\ldots,v_n$, $C_n=A_1A_2\cdots A_n$, where for each $i$ you choose $g$ if all the arrows are pointing into $v_i$ before the $i$th mutation, and $r$ otherwise. Also, $C_n$ is invertible.
\end{lemma}

\begin{definition}[$E_i$]
Given a framed generalized quiver $Q$ and a sequence of mutations $v_1,\ldots,v_i$, let $\bm{E_i^g}$ be the identity matrix with entry $(v_i,v_i)$ replaced with $-1$ and for every arrow pointing from $v_i$ to another base quiver vertex $u$ attached to $v_i$, add 1 to the entry $(v_i,u)$. $\bm{E_i^r}$ is defined analogously with arrows pointing into $v_i$. We will just refer to the relevant $E_i^g$ or $E_i^r$ by $E_i$, where the relevant one is what corresponds to the color of $r_i$ as will be defined soon. The other one will be $E_i^*$.
\end{definition}

\begin{definition}[$D$-Matrix]
The \textbf{$\bm{D}$-matrix} is similar to the $C$-matrix, defined slightly differently though: $D_n=E_1\cdots E_n$. Note that $D_n$ is also invertible. The $D$-matrix between two steps is defined as follows. For two indices $m$ and $n$, let $D_{m,n}$ be $D_m^{-1}D_n$. By extension, $D_n=D_{0,n}$.
\end{definition}

\begin{definition}[$r_i$] \label{ri}
Given the quiver formed by a framed quiver $Q$ and a mutation sequence $v_1,\ldots,v_i$, let $\bm{r_i}$ be the monomial that is the product of the frozen nodes pointing into/from vertex $v_i$. More precisely, we have $r_i = \prod_{j = 1}^{v} y_j^{\delta C_i[j,v_i]|}$ where $\delta \in \{\pm 1\}$ chosen so that the exponents are nonnegative. The $r_i$ is called green if the frozen nodes are pointing in before the final mutation (out afterwards) and red if they are all pointing out. 
\end{definition}

\begin{theorem}[Sign Coherence \cite{GHKK}]
Given the quiver formed by a framed quiver $Q$ and a mutation sequence $v_1,\ldots,v_i$, all the frozen vertices pointing into any vertex $u$ at the end of the process are in the same direction (all in or all out), and there will be at least one arrow. Thus, every vertex will be green or red.
\end{theorem}

\begin{definition}[$\delta_{i,j}$]
Let $\delta_{i,j}$ be 1 if $r_i$ and $r_j$ are the same color, and $-1$ otherwise. Mutation 0 counts as red for the purposes of assigning $\delta_{0,i}$. The sign $\delta_{0,i}$ agrees with the sign $\delta$ from Definition \ref{ri}.
\end{definition}

We call a mutation sequence green if all the mutations are green. We often just say the quiver is green if it is clear what mutation sequence we are referring to.

The next two definitions bare no obvious significance, but simplify expressions in future sections. Loosely, the particular values $a_{i,j}$ and $b_{i,j}$ are important because they are the building blocks of coefficients of monomials in the $F$-polynomials.

\begin{definition}[$a_{i,j}$ and $b_{i,j}$]
Given a framed quiver $Q$ and a mutation sequence $v_1,v_2,\ldots,$, for two indices $i \leq j$, let $\bm{a_{i,j}} = D_{i,j}^{-1}[v_j,v_i]$. Similarly, let $\bm{b_{i,j}} = E^*_jE_jD_{i,j}^{-1}[v_j,v_i]$, or equal to 0 when $i=j$.
\end{definition}

\begin{definition}[$W$]
Given a framed quiver $Q$ and a mutation sequence $v_1,\ldots,v_n$, and a sequence of indices $w_1,\ldots,w_k$ where the $w_i$'s are all between 1 and $n$ inclusive, let $$\bm{W(n, w_1,\ldots,w_k)}=\prod_{i=1}^{k} \left( a_{w_i,n} + \sum_{j=i+1}^{k} (-a_{w_i,w_j}+b_{w_i,w_j})\right).$$
\end{definition}

\begin{definition}[$\phi$]
We define $\bm{\phi(\sf{w})}$ to be the number of distinct permutations of $\sf{w}=w_1,\ldots,w_k$ divided by $k!$. Note that this is just $\frac{1}{\prod{a_i!}}$ where the $a_i$ are the number of occurrences of each of the distinct values taken on by $w_1,\ldots,w_k$.
\end{definition}

\subsection{Simplifications in the Skew-symmetric Case} \label{simpledefs}
In the case where $B$ is skew-symmetric, for any edge between two base vertices, the number of edges attached to each vertex is the same. As such, we can define the quiver, rather than generalized quiver, so that only one set of edges is included, and attached can be disregarded. The analogous definitions are straightforward, but we redefine for completeness. To understand that all proofs completed in the general case also apply here, one should check that all the variables defined, $C$, $A$, $a_{i,j}$, $b_{i,j}$, $r_i$, etc. are all the same when returning the quiver to its generalized quiver form.

\begin{definition}[Quiver]
A \textbf{quiver} $Q$ is a directed graph, along with a Laurent polynomial label at every vertex. The variable $v$ denotes the number of vertices in $Q$, or $|Q|$. The directed graph must have no 2-cycles.
\end{definition}

Any quiver can be viewed, instead, as a generalized quiver by replacing each edge with two edges, one attached to each incident vertex.

\begin{definition}[Quiver Mutation]
An \textbf{quiver mutation} at vertex $i$ is defined by changing both the quiver around vertex $i$ and the Laurent polynomial at vertex $i$. The quiver changes are defined as follows.
\begin{itemize}
\item For every path $j \to i \to k$, draw a new edge from $j$ to $k$.
\item For every path $j \to i \to k$, draw a new edge from $j$ to $k$. Reverse the direction of every edge pointing with $i$ as a vertex.
\item Delete all 2-cycles.
\end{itemize}
The Laurent polynomial is replaced as follows. Let $V_k$ denote the Laurent polynomial currently on vertex $k$. Replace $V_i$ with the following polynomial $P$. $$P=\frac{\prod_{j \in S_1}V_j + \prod_{j \in S_2}V_j}{V_i}$$ where $S_1$ is the multiset of edges pointing into $i$ attached to $i$ and $S_2$ is the multiset of edges pointing out of $i$.
\end{definition}

\begin{definition}[Framed Quiver]
The \textbf{framed quiver} of a quiver $Q$ with labels $1$ at every vertex is the quiver formed by adding a vertex $i'$ for every vertex $i$ with an edge pointing from $i'$ to $i$. The new vertices $i'$ have the label $y_i$. A \textbf{base vertex} of the framed quiver is a quiver that is also a vertex of the original quiver. 
\end{definition}

\begin{definition}[$C$-matrix]
The \textbf{$\bm{C}$-matrix} of the quiver formed by a framed quiver $Q$ and mutation sequence $v_1,\ldots,v_n$ with $v$ base vertices is the $v \times v$ matrix with entry $(i,j)$ as the number of edges from frozen vertex $i'$ to base vertex $j$. It is denoted $C_n$. The $C$-matrix between two steps is defined as follows. For two indices $m$ and $n$, let $\bm{C_{m,n}}$ be $C_m^{-1}C_n$. It is worth noting that $C_n=C_{0,n}$.
\end{definition}

\begin{definition}[$A_i$]
Given a framed quiver $Q$ and a sequence of mutations $v_1,\ldots,v_i$, let $\bm{A_i^g}$ be the identity matrix with entry $(v_i,v_i)$ replaced with $-1$ and for every arrow pointing from $v_i$ to another base quiver vertex $u$, add 1 to the entry $(v_i,u)$. $\bm{A_i^r}$ is defined analogously with arrows pointing into $v_i$. We will just refer to the relevant $A_i^g$ or $A_i^r$ by $A_i$, where the relevant one is what corresponds to the color of $r_i$ as will be defined soon. The other one will be $A_i^*$.
\end{definition}

The $A$-matrix is the same as the $E$-Matrix. Both will be denoted by $A$. The $C$-matrix is the same as the $D$-matrix. Both will be denoted by $C$. Also note the following:

\begin{definition}[$a_{i,j}$ and $b_{i,j}$]
Given a framed quiver $Q$ and a mutation sequence $v_1,v_2,\ldots$, for two indices $i \leq j$, let $\bm{a_{i,j}} = C_{i,j}^{-1}[v_j,v_i]$. Similarly, let, $\bm{b_{i,j}} = A^*_jA_jC_{i,j}^{-1}[v_j,v_i]$, or equal to 0 when $i=j$.
\end{definition}

\begin{definition}[$W$]
Given a framed quiver $Q$ and a mutation sequence $v_1,\ldots,v_n$, and a sequence of indices $w_1,\ldots,w_k$ where the $w_i$'s are all between 1 and $n$ inclusive, let $$\bm{W(n, w_1,\ldots,w_k)}=\prod_{i=1}^{k} \left( a_{w_i,n} + \sum_{j=i+1}^{k} (-a_{w_i,w_j}+b_{w_i,w_j}) \right).$$
\end{definition}

We will use $a_{i,j}$, $b_{i,j}$ and $W$ regardless of whether we are dealing with a quiver or generalized quiver. 

\subsection{Definitions for Stabilization}

\begin{definition}[Matrix Action on Polynomial]
Given an $v\times v$ matrix $M$ and a polynomial $P$ in $x_1,\ldots,x_v$, we let $\bm{M(P)}$ denote the polynomial formed when $M$ acts on the vector of exponents for each monomial in $P$.
\end{definition}

\begin{definition}[Deformed $F$-Polynomial]
The $n$th \textbf{$\bm{S}$-polynomial} (deformed $F$-polynomial) of a framed quiver $Q$ with a mutation sequence $v_1,\ldots,v_n$ is defined as $-C^{-1}(F_n)$.
\end{definition}

\begin{definition}[Fundamental $r_i$]
An $r_i$ of a framed quiver $Q$ and mutation sequence $v_1,v_2,\ldots$ is called \textbf{fundamental} if it is not expressible as the product of at least 2 other $r_j$'s with $j< i$. Let the set $M_n$ denote the set of fundamental $r_i$ for a framed quiver $Q$ and sequence of mutations $v_1,\ldots,v_n$. $M$ is the set of fundamental $r_i$ for the entire mutation sequence $v_1,v_2,\ldots$.
\end{definition}

\begin{definition}[Basic]
Given a framed quiver $Q$ and a mutation sequence $v_1,\ldots,v_n$, a monomial $m$ is called \textbf{basic} if its coefficient is nonzero in a polynomial that appears on one of the vertices after $n$ mutations. In other words, its coefficient is nonzero in some element of the cluster. Let $P_n$ denote the set of basic monomials after $n$ mutations.
\end{definition}

\section{$\bm{F}$-Polynomial Formula}

\subsection{Derivation}

\begin{theorem}\label{formula}
Given a framed quiver or a framed generalized quiver $Q$ and a mutation sequence $v_1,\ldots,v_n$, $F_n$ can be calculated as follows. Let $W$ be the set of (possibly empty) sequences $\sf{w}$ with $1\leq w_1 \leq \ldots \leq w_k\leq n$.
\begin{equation}
F_n = \sum_{\sf{w}\in W} \phi(\sf{w}) W(n, w_1,\ldots,w_k) \prod_{i=1}^{k}r_{w_i}
\end{equation}
Consequently, the coefficient of a given monomial $p$ is the sum of all sequences $1\leq w_1 \leq \ldots\leq w_k \leq n$ so that $$p = \prod_{i=1}^{k}r_{w_i}$$ of $$ \phi(\sf{w}) W(n, w_1,\ldots,w_k).$$ (Recall that $r_{w_i}$ is a monomial in $y_1,\dots,y_v$ as in Definition \ref{ri}.)
\end{theorem}

\begin{proof}
It suffices to prove the statement in the case of framed generalized quivers. 

Consider the steps $v_1,\ldots,v_n$ of the mutation process on the quiver $Q$ that results in $F_n$. At each step $i$, define $\succ_i$ and $\prec_i$ as follows. If $r_i$ is green, any edge in the quiver before the $i$th mutation attached to $a$ denoted $a\to b$ is now denoted $a \succ_i b$. Similarly, $a \leftarrow b$ is denoted $a \prec_i b$. If $r_i$ is red, the reverse occurs: $\succ_i$ replaces $\leftarrow$, and $\prec_i$ replaces $\to$. Thus, $\succ$ replaces $\to$ or $\leftarrow$ corresponding to the direction of the frozen vertices into $v_i$. At each step $i$, the following occurs to calculate $F_i$. We have $$F_iV_{v_i}=\prod_{v_i \succ j}V_j+r_i\prod_{v_i \prec j}V_j.$$ We modify this mutation process: instead of plugging in the monomial $r_i$, replace $r_i$ with the variable $z_i$. Define the analog of $F_i$ as $Z_i$ and the analog of $V_j$ as $ZV_j$, so $ZV_j$ is the $Z$ power series at the $j$th vertex. Thus, $Z_n$ will be some rational function in the $z_i$'s, computed by the following recurrence relation: $$Z_iZV_{v_i}=\prod_{v_i \succ j}ZV_j+z_i\prod_{v_i \prec j}ZV_j$$ and $\ldots,F_{-1},F_0=1$. When we plug in $r_i$ for each $z_i$, we get back $F_n$. As such, to get the coefficient of a given monomial $p$, do the following: express $Z_n$ in its MacLaurin series form, i.e. an element of $\C[[z_1,\ldots,z_n]]$. Then, add the coefficients over all sequences $w_1,\ldots,w_k$ so that $p=r_{w_1}\cdots r_{w_k}$.

Now, we compute the coefficient of a given $z_{w_1}\cdots z_{w_k}$. We show it is equal to $$\phi(w_1,\ldots,w_k) W(n, w_1,\ldots,w_k).$$ Encode the sequence $\sf{w}$ instead as two sequences $\sf{q}$ and $\sf{m}$ where $\sf{q}$ consists of the distinct terms of $\sf{w}$ in increasing order, and $\sf{m}$ is the corresponding multiplicity of each. Let $l=|\sf{q}|$. The coefficient of the term $$z_{w_1}\cdots z_{w_k}=z_{q_1}^{m_1}\cdots z_{q_l}^{m_l}$$ in $Z_n$ is the same as the one in the power series computed as follows: look at the entire modified mutation process, except on a mutation step $i$ not corresponding to a $q_j$, replace $z_i$ with 0. This is because replacing all irrelevant $z_i$'s with 0 does not affect the coefficient of the desired term. Call this new power series $Y_n$. $YV_{v_i}$ is defined analogously as well. Consider the following power series, $L_1,\ldots,L_l$.
$$
\begin{aligned}
L_1&=(1+z_{q_1}) \\
L_2&=\left(1+z_{q_2}(1+z_{q_1})^{-a_{q_1,q_2}+b_{q_1,q_2}}\right) \\
&. \\
&. \\
&. \\
L_l&=\left(1+z_{q_l}(L_{l-1}^{-a_{q_{l-1},q_l}+b_{q_{l-1},q_l}}\cdots L_1^{-a_{q_1,q_l}+b_{q_1,q_l}})\right)
\end{aligned}
$$
Let $$H_i:=\prod_{j=1}^{r}L_j^{a_{q_j,i}}.$$ where $r$ corresponds to the final $q_r$ preceding $i$ (so $H_n$ would be $\prod_{j=1}^{l}L_j^{a_{q_j,i}}$).
We show that at any given step $i$, $Y_i=H_i$. We do so by induction.

\paragraph{Base Case: $\bm{i\leq q_1}$:} For this case, we have $$Y_{q_1}=H_{q_1}=(1+z_{q_1})$$ and for $i<q_1$, we have $$Y_i=H_i=1.$$ $Y_i$ is 1 because all the relevant $z_i$ are 0, and $H_i$ is 1 because it is the product over no $L_j$ and therefore just 1.

We have the following recurrence relation. $$Y_iYV_{v_i}=\prod_{v_i \succ j}YV_j+z_i\prod_{v_i \prec j}YV_j$$ I claim that the relation $$H_iHV_{v_i}=\prod_{v_i \succ j}HV_j+z_i\prod_{v_i \prec j}HV_j$$ holds as well. Because of the inductive hypothesis, proving this recurrence relation would suffice to show the power series are equal.

\paragraph{Case 1: $\bm{i\neq q_r}$:} We have $z_i=0$, so we want to show$$
\begin{aligned} H_iHV_{v_i}&=\prod_{v_i \succ j}HV_j+z_i\prod_{v_i \prec j}HV_j \\
\iff H_i&=\frac{\prod_{v_i \succ j}HV_j}{HV_{v_i}}.
\end{aligned}$$

We look at the exponent for each $L_s$ in $H_i$ and show that it is, in fact, the same on both sides of the equation. The exponent on the right hand side, denoted $e_s\left( \frac{\prod_{v_i \succ j}HV_j}{HV_{v_i}} \right)$ is calculated as follows: $$\sum_{v_i \succ j}e_s(HV_j)-e_s(HV_{v_i}).$$ Let $i_j$ be the last index of a mutation at vertex $j$. By the inductive hypothesis, $$e_s\left( \frac{\prod_{v_i \succ j}HV_j}{HV_{v_i}} \right)=\sum_{v_i \succ j}a_{q_s,i_j}-a_{q_s,i_{v_i}}.$$ Recall that $$D_{i,j}=E_{i+1}\cdots E_j \implies D^{-1}_{i,j}=E_j\cdots E_{i+1}$$ since the $E_i$ are involutions. Multiplication by any $E_i$ only changes the corresponding row, so since $a_{i,j}=D^{-1}_{i,j}[v_j,v_i]$, we can replace terms of the form $a_{q_s,i_j}$ with $D^{-1}_{q_s,i-1}[j,v_{q_s}]$, since no entry in the $j$th row of the corresponding matrix changes between step $i_j$ and step $i-1$. It's useful to note this also stays 0 when $L_s$ corresponds to a term that appears after one of our $HV$ terms is calculated. Thus, we have as our expression: $$e_s\left( \frac{\prod_{v_i \succ j}HV_j}{HV_{v_i}} \right)=\sum_{v_i \succ j}D^{-1}_{q_s,i-1}[j,v_{q_s}] -D^{-1}_{q_s,i-1}[v_i,v_{q_s}].$$ However, the right hand side is exactly the operation corresponding to obtaining the entry $D^{-1}_{q_s,i}[v_i,v_{q_s}]=a_{q_s,i}$ from how $E_i$ is defined: add the contributions from all the vertices pointing opposite of the frozen vertices, and subtract the current vertex's contribution. This exactly proves our desired statement that $$e_s(H_i)=e_s\left( \frac{\prod_{v_i \succ j}HV_j}{HV_{v_i}} \right).$$

\paragraph{Case 2: $\bm{i=q_r}$:} Just as in the last case, we can try to compute $$\frac{\prod_{v_i \succ j}HV_j+z_i\prod_{v_i \prec j}HV_j}{HV_{v_i}}.$$ By the first case, $$\frac{\prod_{v_i \succ j}HV_j}{HV_{v_i}}$$ is $H_i$ without the contribution from $L_r$. We now have to show that adding $\frac{z_i\prod_{v_i \prec j}HV_j}{HV_{v_i}}$ gives the same contribution as multiplying by $L_r$. By the same logic as the proof from the first case, we find the term $$\frac{\prod_{v_i \prec j}HV_j}{HV_{v_i}}.$$ Again we find $e_s$ of the expression for each $k$. We get $$e_s\left(\frac{\prod_{v_i \prec j}HV_j}{HV_{v_i}}\right)=\sum_{j \succ v_i}D^{-1}_{q_s,i-1}[j,v_{q_s}] -D^{-1}_{q_s,i-1}[v_i,v_{q_s}].$$ This is exactly the operation of $E^*_i$, so we get $$e_s\left(\frac{\prod_{j \succ v_i}HV_j}{HV_{v_i}}\right)=b_{q_j,i}.$$ Now that we have these identities, we can show that our recurrence relation in fact gives $H_i$.
$$\begin{aligned}
&\frac{\prod_{v_i \succ j}HV_j+z_i\prod_{v_i \prec j}HV_j}{HV_{v_i}} \\
=&\prod_{j=1}^{r-1}L_j^{a_{q_j,i}}+z_i\prod_{j=1}^{r-1}L_j^{b_{q_j,i}} \\
=&(1+z_i\prod_{j=1}^{r-1}L_j^{b_{q_j,i}-a_{q_j,i}})\prod_{j=1}^{r-1}L_j^{a_{q_j,i}} \\
=&L_r\prod_{j=1}^{r-1}L_j^{a_{q_j,i}} \\
=&H_i
\end{aligned}$$
which is what we desired.

Now we must actually compute the coefficient of $z_{w_1}\cdots z_{w_k}$ in $H_n=\prod_{j=1}^{k}L_j^{a_{q_j,n}}$. We prove that it is $\phi(w_1,\ldots,w_k) W(n, w_1,\ldots,w_k)$ by induction on $l=|\sf{q}|$. 

\paragraph{Base Case: $|\sf{q}|=1$:} Then $H_n=L_1^{a_{q_1,n}}$. By the extended binomial theorem, the coefficient of $z_{q_1}^{m_1}$ is in fact $$\frac{(a_{q_1,i})(a_{q_1,i}-1),\ldots,(a_{q_1,i}-m_1+1)}{m_1!}.$$
\paragraph{Inductive step: $\bm{|}\sf{q}\bm{|=l}$:} If we instead compute the coefficient of the term $z_{q_2}^{m_2}\cdots z_{q_l}^{m_l}$, then we can apply the inductive hypothesis to get the desired result. The corresponding $L_i$, which we call $L_i'$, we have are defined as follows.

$$
\begin{aligned}
L_2'&=(1+z_{q_2}) \\
L_3'&=\left(1+z_{q_3}(1+z_{q_2})^{-a_{q_1,q_2}+b_{q_1,q_2}}\right) \\
&. \\
&. \\
&. \\
L_l'&=\left(1+z_{q_l}(L_{l-1}'^{-a_{q_{l-1},q_l}+b_{q_{l-1},q_l}}\cdots L_1'^{-a_{q_1,q_l}+b_{q_1,q_l}})\right)
\end{aligned}
$$

This gives $H'_n=\prod_{j=1}^{l}L_j'^{a_{q_j,n}}$. By inspecting the way the $L_i$'s are formed, we see that we can form $H_n$ by replacing every instance of $z_{q_i}$ with $z_{q_i}(1+z_{q_1})^{-a_{q_1,q_i}+b_{q_1,q_i}}$, and multiplying the entire expression by $(1+z_{q_1})^{a_{q_1,n}}$. If we expand only $q_2,\ldots,q_l$, leaving the $(1+z_{q_1})$ expressions, which corresponds to just expanding $H'_n$, the term $$z_{q_2}^{m_2}\cdots z_{q_l}^{m_l}$$ corresponds to the term $$(1+z_{q_1})^{a_{q_1,n}+\sum_{i=2}^{l}-m_ia_{q_1,q_i}+m_ib_{q_1,q_i}}z_{q_2}^{m_2}\cdots z_{q_l}^{m_l}.$$ Any other terms cannot have a monomial that has the correct number of factors of all of $z_{q_2},\ldots,z_{q_l}$. Thus, we only need to find the coefficient of $z_{w_1}\cdots z_{w_k}$ in this monomial. The corresponding term in $H'_n$ already has coefficient $\phi(w_{m_1+1},w_k)W(n,w_{m_1+1},\ldots,w_k)$, so the monomial itself is that times $$(1+z_{q_1})^{a_{q_1,n}+\sum_{i=2}^{l}-m_ia_{q_1,q_i}+m_ib_{q_1,q_i}}z_{q_2}^{m_2}\cdots z_{q_l}^{m_l}.$$ To find the total coefficient, we multiply this by the coefficient of $z_{q_1}^{m_1}$ in the expansion of $$(1+z_{q_1})^{a_{q_1,n}+\sum_{i=2}^{l}-m_ia_{q_1,q_i}+m_ib_{q_1,q_i}}$$ which by extended binomial theorem is $$\frac{a_{q_1,n}+\sum_{i=2}^{l}-a_{q_1,q_i}+b_{q_1,q_i} \cdots a_{q_1,n}+\sum_{i=2}^{l}-m_ia_{q_1,q_i}+m_ib_{q_1,q_i}-m_1+1}{m_1!}.$$ Since $a_{i,i}$ is 1 and $b_{i,i}$ is 0, this can be rewritten as $$\frac{\prod_{j=1}^{m_1}\left(a_{w_1,n}+\sum_{i=m_1+1}^{k}-a_{w_j,w_i}+b_{w_j,w_i} \right)}{m_1!}.$$ This, in turn, is exactly $\phi(w_1,\ldots,w_k) W(n, w_1,\ldots,w_k)$, thus proving our inductive step.
\end{proof}

\subsection{Formula Discussion}

In this section, we discuss the formula, its terms, and implications.

The formula has the unfortunate property that it is summing over an infinite set of sequences. However, this is easily remedied; it is fairly simple to find an upper bound on the degree of $F_n$. For example, take the maximum number of edges in the quiver at any stage $m$ and raise this value to the $n$th power. This must be an upper bound, because at each step the maximum degree of a label on the quiver must be less than $m$ times the previous maximum value. Then, simply sum only over sequences where the expression $\prod_{i=1}^{k}r_{w_i}$ has degree less than our upper bound. This is clearly finite, since each $r_{w_i}$ is positive. Consequently, this at least limits us to only checking sequences with length less than the upper bound on the degree. In the next few sections, we apply our formula to get simpler formulas for specific cases where we know what the relevant entries of the $C$-matrix are. In many of the cases, we do not limit the sequences we are summing over, if there isn't a specifically interesting way to do so. It is therefore valuable to keep in mind that this procedure would work.

Secondly, let us look at the full expanded version of the formula and understand what the terms mean. We have for sequences $\sf{w}$ with $1\leq w_1 \leq \ldots \leq w_k\leq n$,
\begin{equation} \label{formulafull}
\begin{aligned}
F_n = \sum_{\sf{w}\in W} & \phi(w_1,\ldots,w_k) \prod_{i=1}^{k} \left( D_{w_i,n}^{-1}[v_n,v_{w_i}] + \sum_{j=i+1}^{k} (-D_{w_i,w_j}^{-1}[v_{w_j},v_{w_i}]+A^*_{w_j}A_{w_j}D_{w_i,w_j}^{-1}[v_{w_j},v_{w_i}])\right) \\
&\prod_{i=1}^{v}y_i^{\sum_{j=1}^{k}\delta_{0,w_j}C_{0,w_j}[i,v_{w_j}]}.
\end{aligned}
\end{equation}

First of all, we should understand which sequences $\sf{w}$ contribute to a given monomial term. These would be the terms that have the same $\prod_{i=1}^{k}r_{w_i}$, or align for each $i$ on $$\sum_{j=1}^{k}\delta_{0,w_j}C_{0,w_j}[i,v_{w_j}].$$ Though this appears complicated to work with, $C$-matrix entries are often easy to understand, so it becomes simple to categorize which sequences contribute to a given monomial term.

Now, let us understand what the $D$-matrix entries mean in the skew-symmetric case, where they are the same as the $C$-matrix entries. Recall that $C_n[i,j]$ is the number of arrows pointing into vertex $j$ from vertex $i'$ after $n$ mutation steps on the mutation sequence used to obtain $C_n$. Similarly, $C_{m,n}[i,j]$ is the number of arrows pointing into vertex $j$ from vertex $i'$ after $n-m$ mutation steps where the initial base quiver is actually the original quiver after $m$ mutation steps. There is a caveat though: the color sequence must be the same as it was on the original quiver. If not, we can force it to be the same, loosely by turning a red mutation into a negative analog of a green mutation and vice versa. $C_{m,n}[i,j]$ also makes sense for $n<m$ (simply do the mutations from step $n$ to $m$ in reverse), so this lets us understand what the entries of inverse matrices mean, since $C_{m,n}[i,j]=C^{-1}_{n,m}[i,j]$.

Another thing to understand is the following. The formula does not do a good job of highlighting which sequences of $r_i$'s are going to contribute negatively and positively to the coefficient of a given monomial. There might be a way to group the sequences, so each group of sequences will be a nonnegative contributor. It would be very cool if we could recover positivity using this formula. It could also help simplify the formula, by reducing the set of sequences we have to sum over, if we know certain groups give 0.

Now we state the analog of the $F$-polynomial formula for deformed $F$-polynomials.

\begin{corollary}
Given a framed quiver $Q$ and a mutation sequence $v_1,\ldots,v_n$, we have that $S_n$ can be calculated as follows. Let $W$ be the set of (possibly empty) sequences $\sf{w}$ with $0\leq w_1 \leq \ldots \leq w_k\leq n-1$.
\begin{equation}
S_n = \sum_{\sf{w}\in W} \phi(\sf{w})W(n,n-w_k,\ldots,n-w_1) \prod_{i=1}^{k}C^{-1}_n(r_{w_i})
\end{equation}
Consequently, for a given monomial $p$, over all sequences $0\leq w_1 \leq \ldots\leq w_k \leq n-1$ so that $$p = \prod_{i=1}^{v}y_i^{\sum_{j=1}^{k}\delta_{n-w_j,n}C^{-1}_{n-w_j,n}[i,v_{n-w_j}]}$$ the value of the coefficient is the sum over all the sequences of $$ \phi(\sf{w}) W(n,n-w_k,\ldots,n-w_1).$$
\end{corollary}

\begin{proof}
This is Theorem \ref{formula}, with the $y_i$ exponents transformed according to the stabilization rule.
\end{proof}

This is interesting when we consider the full form of the formula in the skew symmetric case, namely $$
\begin{aligned}S_n = &\sum_{0\leq w_1\ldots w_k\leq n-1} \phi(w_1,\ldots,w_k)\prod_{i=1}^{k} \bigg((a_{n-w_i,n}) + \sum_{j=1}^{i-1} (-C^{-1}_{n-w_j,n-w_i}[v_{n-w_i},v_{n-w_j}]+ \\
&A^*_{n-w_i}A_{n-w_i}C^{-1}_{n-w_j,n-w_i}[v_{n-w_i},v_{n-w_j}])\bigg) \prod_{i=1}^{v}y_i^{\sum_{j=1}^{k}\delta_{n-w_j,n}C^{-1}_{n-w_j,n}[i,v_{n-w_j}]}.
\end{aligned}$$ The terms $C^{-1}_{n-w_j,n}[i,v_{n-w_j}]$ appear both in the exponents of the $y_i$ terms and as a part of the coefficient calculation. This alludes to why $C_n^{-1}$ is a natural operation, and also why it might cause convergence.

Finally, if we look at the proof of the theorem, recall that we had a definition of $L_i$ and $$H_n:=\prod_{j=1}^{r}L_j^{a_{q_j,n}}$$

where $q_1...q_l$ represent the sequence of $z_i$ (the $r_i$) that are not replaced by 0. This gives a polynomial that aligns with the $F$-polynomial on all terms where only the $r_{q_j}$ are considered. In the case where $q_1=1\ldots q_n=n$ is our sequence, when there are $n$ mutation steps, this process gives the $F$-polynomial exactly. More specifically, we obtain

$$
\begin{aligned}
L_1&=(1+z_1) \\
L_2&=\left(1+z_2(1+z_1)^{-a_{1,2}+b_{1,2}}\right) \\
&. \\
&. \\
&. \\
L_n&=\left(1+z_n(L_{n-1}^{-a_{n-1,n}+b_{n-1,n}}\cdots L_1^{-a_{1,n}+b_{1,n}})\right)
\end{aligned}
$$
and 
$$F_n=\prod_{j=1}^{n}L_j^{a_{j,n}}$$

The author was informed by Dr. Bernhard Keller that for the skew-symmetric case, this is analogous to a formula found by Nagao in \cite{N2}.

Finally, we demonstrate an example of our main formula in theorem \ref{formula}.

\begin{example}\label{k2}
Let $K_2$ be the quiver with vertices 1 and 2, and two arrows from vertex 1 to vertex 2. Consider the mutation sequence $1,2,1$. We aspire to find $F_3$ given this sequence. Our formula states $$F_3 = \sum_{\sf{w}\in W} \phi(\sf{w}) W(3, w_1,\ldots,w_k) \prod_{i=1}^{k}r_{w_i}$$ where $W$ consists of nondecreasing sequences of 1,2,3's. As alluded earlier, we can simplify this set by only considering sequences such that $\prod_{i=1}^{k}r_{w_i}$ has degree at most the degree of $F_3$ which is $3$ in $y_1$ and $2$ in $y_2$. Noting that $r_1=y_1, r_2=y_1^2y_2, r_3=y_1^3y_2^2$, the sequences in $W$ are $[1], [1,1], [1,1,1], [1,2], [2], [3]$. Thus, we obtain (shortening the input to $W$ fixing $n=3$), $$F_3 = \phi(1)W(1)r_1 + \phi(1,1)W(1,1)r_1^2 + \phi(1,1,1)W(1,1,1)r_1^3 + \phi(1,2)W(1,2)r_1r_2 + \phi(2)W(2)r_2 + \phi(3)W(3)r_3.$$ To compute the $W$ values, we need to know the relevant $a_{i,j}$ and $b_{i,j}$ values. We know $a_{i,i}$ and $b_{i,i}$ to be 1 and 0 respectively, so we compute the others. These are computed in terms of $C$-matrix entires. Importantly, we have $$
\begin{aligned}
A^g_i&=\begin{pmatrix} -1 & 2 \\ 0 & 1 \end{pmatrix} \text{for $i$ odd} \\
A^g_i&=\begin{pmatrix} 1 & 0 \\ 2 & -1 \end{pmatrix} \text{for $i$ even} \\
A^r_i&=\begin{pmatrix} -1 & 0 \\ 2 & 1 \end{pmatrix} \text{for $i$ odd} \\
A^r_i&=\begin{pmatrix} 1 & 2 \\ 0 & -1 \end{pmatrix} \text{for $i$ even}.
\end{aligned}$$ Recall $a_{i,j}=C^{-1}_{i,j}[v_j,v_i]$. This yields $a_{1,2}=2, a_{1,3}=3, a_{2,3}=2$. Also recall $b_{i,j} = A^*_jA_jC_{i,j}^{-1}[v_j,v_i]$, which yields $b_{1,2}=0, b_{1,3}=1, b_{2,3}=0$. Recall that $$W(n, w_1,\ldots,w_k)=\prod_{i=1}^{k} \left( (a_{w_i,n}) + \sum_{j=i+1}^{k} (-a_{w_i,w_j}+b_{w_i,w_j})\right).$$ Plugging in, we get $W(1)=3$, $W(1,1)=6$, $W(1,1,1)=6$, $W(1,2)=2$, $W(2)=2$, and $W(3)=1$. Noting $\phi$ values, this yields $$F_3 = 3y_1 + \frac{1}{2}6y_1^2 + \frac{1}{6}\cdot 6 y_1^3 + 2y_1^3y_2 + 2y_1^2y_2 + y_1^3y_2^2.$$ This simplifies to the correct expression of $$F_3 = 3y_1 + 3y_1^2 + y_1^3 + 2y_1^3y_2 + 2y_1^2y_2 + y_1^3y_2^2.$$
\end{example}

\section{Exact $F$-Polynomials}

\subsection{Symmetric Quivers}
We aspire to compute $F_n$ for graphs with useful symmetry properties. We then use this to directly do the next two cases.

\begin{definition}
Let a Symmetric Quiver $Q$ be one with the following properties. Take a quiver $Q$ with base quiver $B$ and the periodic mutation sequence $1,\ldots,v,1,\ldots$, with the following properties.
\begin{itemize}
\item The base quiver $B$ is reversible; that is reversing all the arrows and relabeling vertex $i$ with $v+1-i$ returns the same quiver.
\item The sequence of mutations is cyclic; that is mutating $B$ at 1 returns the original quiver upon relabeling vertex $i$ with $i-1$, where 1 is relabeled with $v$.
\item The base quiver has a symmetrical structure; that is swapping every vertex $2 \leq i \leq v$ with $v+2-i$ and fixing vertex 1 returns a quiver with the same arrows pointing in and out of vertex $1$.
\item The quiver is entirely green.
\end{itemize}
For the quiver $Q$ define $$f(x_v,\ldots,x_1)=-x_1+\sum_{\text{ edge in } B \text{ } 1 \to i}x_i$$ and $$g(x_v,\ldots,x_1)=-x_1+\sum_{\text{ edge in } B \text{ } i \to 1}x_i.$$ Define the recurrence relation $s_i=f(s_{i-1},\ldots,s_{i-v})$ and $\ldots,s_{-2}=s_{-1}=0, s_0=1$. Let $s'_{i}=g(s_{i-1},\ldots,s_{i-v})$ 
\end{definition}

\begin{lemma}\label{symr}
Take a symmetric quiver $Q$ and mutation sequence $1,\ldots,v,1,\ldots$. with the definitions above. Then, $$r_k=\prod_{i=1}^{v}y_i^{s_{k+1-i}}.$$
\end{lemma}

\begin{proof}
We prove the statement by induction.

\paragraph{Base Case: step $\bm{k=1}$: }We have $r_i=y_1$ and $\prod_{i=1}^{v}y_i^{s_{1-i}}=y_1$. 

\paragraph{Inductive Step: step $\bm{k \neq 1}$: } We are mutating at mutation step $k$. We need to find the contributors to the frozen vertices pointing into $k$. Consider the previous cycle of mutation steps, and possible arrows pointing into or out of vertex $k$ when the vertex was mutated (thus contributing to vertex $k$) Due to symmetry and cyclicity, the things that were pointing in must be precisely what is pointing out of vertex $k$ now, and vice versa. This easily gives the relation $e_i=f(e_{i-1},\ldots,e_{i-v})$ where $e_i$ is the exponent vector of a given $r_i$. It is worth verifying that this logic also holds when $k<v$ (vertices that have not yet been mutated will have contribution 0). In turn, the relation proves the desired result for $k$ by induction, and so we have proven the lemma.
\end{proof}

\begin{theorem}\label{sym}
Take a symmetric quiver $Q$ and mutation sequence $1,\ldots,v,1,\ldots$. Let $W$ be the set of sequences $\sf{w}$ with $1\leq w_1 \leq \ldots\leq w_k \leq n$. Then $$F_n=\sum_{\sf{w}\in W} \phi(\sf{w}) \prod_{i=1}^{k} \left( s_{n-w_i}+\sum_{j=i}^{k}-s_{w_j-w_i}+s'_{w_j-w_i} \right)\prod_{i=1}^{v}y_i^{\sum_{j=1}^{k}s_{w_j+1-i}}.$$
\end{theorem}

\begin{proof}
We have $$F_n = \sum_{1\leq w_1,\ldots,w_k\leq n} \phi(\sf{w}) \prod_{i=1}^{k} \left( (a_{w_i,n}) + \sum_{j=i+1}^{k} (-a_{w_i,w_j}+b_{w_i,w_j}) \right) \prod_{i=1}^{k}r_{w_i}.$$ Since the quiver is reversible, note that $C_k^{-1}=C_k$ after swapping indices $0$ and $v-1$, $1$ and $v-2$ etc. Furthermore, $C_{i,k}=C_{k-i}$ after cyclic shifting the indices. With a little work, one can therefore see that $a_{i,k}=s_{k-i}$ and similarly $b_{i,k}=s'_{k-i}$. Further, one can see that $a_{0,i}=s_i$ and $b_{0,i}=s'_i$ because of the lemma. Therefore, $$F_n = \sum_{1\leq w_1,\ldots,w_k\leq n} \phi(\sf{w}) \prod_{i=1}^{k} \left( (s_{n-w_i}) + \sum_{j=i+1}^{k} (-s_{w_j-w_i}+s'_{w_j-w_i}) \right) \prod_{i=1}^{k}r_{w_i}$$ By lemma \ref{symr}, we get $$F_n = \sum_{1\leq w_1\ldots w_k\leq n} \phi(\sf{w}) \prod_{i=1}^{k} \left( (s_{n-w_i}) + \sum_{j=i+1}^{k} (-s_{w_j-w_i}+s'_{w_j-w_i}) \right) \prod_{i=1}^{v}y_i^{\sum_{j=1}^{k}s_{w_j+1-i}}$$ which is the desired expression.
\end{proof}

\subsection{$K_r$}

\begin{definition}
The framed quiver $\bm{K_r}$ is the quiver with base quiver having two vertices, 0 and 1, and $r$ arrows from 0 to 1. $K_3$ is illustrated below.

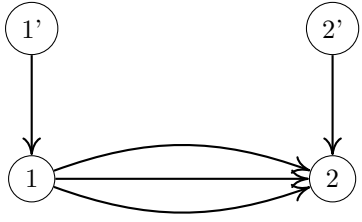
\begin{center}
\begin{tikzpicture}
\node[shape=circle,draw=black] (1) at (0,0) {1};
\node[shape=circle,draw=black] (2) at (4,0) {2};
\draw[-{>[scale length=2]}, thick] (1) to [bend right=20] (2);
\draw[-{>[scale length=2]}, thick] (1) to (2);
\draw[-{>[scale length=2]}, thick] (1) to [bend left=20] (2);

\node[shape=circle,draw=black] (1') at (0,2) {1'};
\node[shape=circle,draw=black] (2') at (4,2) {2'};
\draw[-{>[scale length=2]}, thick] (1') to (1);
\draw[-{>[scale length=2]}, thick] (2') to (2);
\end{tikzpicture}
\end{center}
\end{definition}

\begin{definition}\label{sdef}
Define the sequence $\sf{s}$ as $\ldots,s_{-2}=s_{-1}=0, s_0=1$ and $s_{k+2}=rs_{k+1}-s_k$ for $k\ge -1$.
\end{definition}

The main focus of the papers \cite{L} and \cite{LS} are to find explicit formulas for the $F$-Polynomials of the quivers $K_r$. The former derives an algebraic expression that appears to be more complicated than ours. The latter interprets the coefficients as the count of a combinatorial object, and thus is incomparable to our work. It is interesting that an explicit formula follows as a simple corollary of our general formula.

\begin{theorem}
We can compute $F_n$ for $K_r$ as follows. Let $W$ be the set of (possibly empty) sequences $\sf{w}$ with $1\leq w_1\leq \ldots \leq w_k\leq n$. $$F_n=\sum_{\sf{w}\in W} \phi(\sf{w}) \prod_{i=1}^{k} \left( s_{n-w_i} - \sum_{j=i+1}^{k} s_{w_j-w_i} + s_{w_j-w_i-2} \right)y_1^{\sum_{i=1}^{k}s_{w_i-1}}y_2^{\sum_{i=1}^{k}s_{w_i-2}}$$
\end{theorem}

\begin{proof}
Notice that $K_r$ is a symmetric quiver when $r\geq 2$. Recall Theorem \ref{sym} which states that $$F_n=\sum_{\sf{w}\in W} \phi(\sf{w}) \left( \prod_{i=1}^{k}s_{n-w_i}+\sum_{j=i}^{k}-s_{w_j-w_i}+s'_{w_j-w_i} \right)\prod_{i=1}^{v}y_i^{\sum_{j=1}^{k}s_{w_j+1-i}}.$$ Here, $\sf{s}$ is defined exactly as in \ref{sdef}, and $s'_i=-s_{i-2}$. Plugging in therefore gives $$F_n = \sum_{1\leq w_1\ldots w_k\leq n} \phi(\sf{w}) \prod_{i=1}^{k} \left( (s_{n-w_i}) - \sum_{j=i+1}^{k} (s_{w_j-w_i}+s_{w_j-w_i-2}) \right) y_1^{\sum_{j=1}^{k}s_{w_j-1}}y_2^{\sum_{j=1}^{k}s_{w_j-2}}$$
This is exactly our desired expression.
\end{proof}

Note that the degree in $y_1$ of $F_n$ is $s_{n-1}$, so we only have to sum over sequences with $$\sum_{i=1}^{k}s_{w_i-1} \leq s_{n-1}.$$

Another thing to note is that this language greatly simplifies Example \ref{k2}.

\subsection{Gale Robinson Quivers}

We now aspire to compute $F_n$ for Gale Robinson Quivers specifically.
\begin{definition}
Define the Gale Robinson Quiver $G_{v,r,t}$ as follows.
\begin{itemize}
\item There are $v$ vertices labeled $1,\ldots,v$.
\item For all $1 \leq i \leq v - r$, draw an arrow $i \to i + r$, and for all $1 \leq j \leq r$, draw an arrow
$j \to v-r+j$.
\item  For all $1\leq i \leq v-t$, draw an arrow $t+i\to i$, and for all $1\leq j\leq t$, draw an arrow
$v-s+j\to j$.
\item For all $1\leq i\leq v-r-t$, draw an arrow from $r+i\to t+i$ and for all $1\leq j\leq t-r$, draw an
arrow $r+j\to v-t+j$.
\item Delete any 2-cycles created in the above process.
\end{itemize}
We illustrate $G_{7,2,3}$ below.

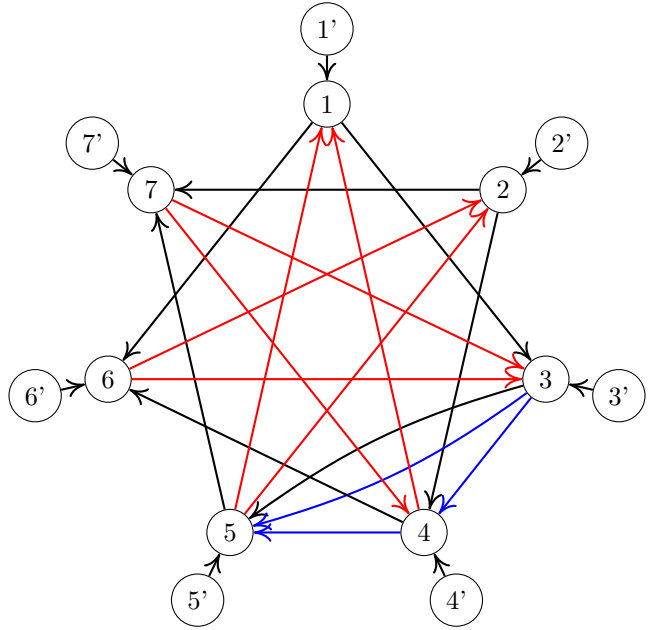
\begin{center}
\begin{tikzpicture}
\node[shape=circle,draw=black] (1) at (0,3) {1};
\node[shape=circle,draw=black] (2) at (2.34,1.86) {2};
\node[shape=circle,draw=black] (3) at (2.91,-0.66) {3};
\node[shape=circle,draw=black] (4) at (1.29,-2.70) {4};
\node[shape=circle,draw=black] (5) at (-1.29,-2.70) {5};
\node[shape=circle,draw=black] (6) at (-2.91,-0.66) {6};
\node[shape=circle,draw=black] (7) at (-2.34,1.86) {7};
\draw[-{>[scale length=2]}, thick] (1) to (3);
\draw[-{>[scale length=2]}, thick] (1) to (6);
\draw[-{>[scale length=2]}, thick] (2) to (4);
\draw[-{>[scale length=2]}, thick] (2) to (7);
\draw[-{>[scale length=2]}, blue, thick] (3) to (4);
\draw[-{>[scale length=2]}, thick] (3) to [bend right=10] (5);
\draw[-{>[scale length=2]}, blue, thick] (3) to [bend left=10] (5);
\draw[-{>[scale length=2]}, red, thick] (4) to (1);
\draw[-{>[scale length=2]}, blue, thick] (4) to (5);
\draw[-{>[scale length=2]}, thick] (4) to (6);
\draw[-{>[scale length=2]}, red, thick] (5) to (1);
\draw[-{>[scale length=2]}, red, thick] (5) to (2);
\draw[-{>[scale length=2]}, thick] (5) to (7);
\draw[-{>[scale length=2]}, red, thick] (6) to (2);
\draw[-{>[scale length=2]}, red, thick] (6) to (3);
\draw[-{>[scale length=2]}, red, thick] (7) to (3);
\draw[-{>[scale length=2]}, red, thick] (7) to (4);

\node[shape=circle,draw=black] (1') at (0,4) {1'};
\node[shape=circle,draw=black] (2') at (3.12,2.48) {2'};
\node[shape=circle,draw=black] (3') at (3.88,-0.88) {3'};
\node[shape=circle,draw=black] (4') at (1.72,-3.60) {4'};
\node[shape=circle,draw=black] (5') at (-1.72,-3.60) {5'};
\node[shape=circle,draw=black] (6') at (-3.88,-0.88) {6'};
\node[shape=circle,draw=black] (7') at (-3.12,2.48) {7'};
\draw[-{>[scale length=2]}, thick] (1') to (1);
\draw[-{>[scale length=2]}, thick] (2') to (2);
\draw[-{>[scale length=2]}, thick] (3') to (3);
\draw[-{>[scale length=2]}, thick] (4') to (4);
\draw[-{>[scale length=2]}, thick] (5') to (5);
\draw[-{>[scale length=2]}, thick] (6') to (6);
\draw[-{>[scale length=2]}, thick] (7') to (7);
\end{tikzpicture}
\end{center}
\end{definition}

\begin{lemma}
$G_{v,r,t}$ has the following properties.
\begin{itemize}
\item The base quiver $B$ is reversible.
\item The sequence of mutations is cyclic.
\item The quiver is entirely green.
\item Only $1+r$ and $v+1-r$ point out of vertex 1 and $1+t$ and $v+1-t$ point into 0 (this implies it is symmetric).
\end{itemize}
\end{lemma}

\begin{proof}
The first statement is easy to check, since the arrows are explicitly stated. The second and third are well-known \cite{JMZ}. The fourth, again, is easy to check since the arrows are explicitly stated.
\end{proof}

\begin{theorem}
Take the quiver $G_{v,r,t}$. Define a sequence $\sf{s}$ with $s_i$ as the number of partitions of $i$ into parts $r,v-r$ (0 for negative values of $i$, 1 for 0). Let $W$ be the set of sequences with $1\leq w_1 \leq \ldots \leq w_k \leq n$. Then
$$F_n=\sum_{\sf{w}\in W} \phi(\sf{w}) \left( \prod_{i=1}^{k}s_{n-w_i}+\sum_{j=i}^{k}-s_{w_j-w_i}-s_{w_j-w_i-v}+s_{w_j-w_i-t}+s_{w_j-w_i-v+t} \right)\prod_{i=1}^{v}y_i^{\sum_{j=1}^{k}s_{w_j-i}}$$
\end{theorem}

\begin{proof}
A Gale Robinson Quiver is in fact a symmetric quiver. Therefore, we can just plug into the formula in the previous section. Define $s_i$ as in the previous section. It is not difficult to see that $s_i$ is the number of partitions of $i$ into parts $r,v-r$ (0 for negative values of $i$, 1 for 0). We also have $s_i'=s_{i-t}+s_{i-v+t}-s_{i-v}$. Thus, literally plugging in to Theorem \ref{sym} gives the desired equation of $$F_n=\sum_{\sf{w}\in W} \phi(\sf{w}) \left( \prod_{i=1}^{k}s_{n-w_i}+\sum_{j=i}^{k}-s_{w_j-w_i}-s_{w_j-w_i-v}+s_{w_j-w_i-t}+s_{w_j-w_i-n+t} \right)\prod_{i=1}^{v}y_i^{\sum_{j=1}^{k}s_{w_j-i}}$$
\end{proof}

Gale Robinson quivers have been specifically studied in depth already. A paper by Jeong, Musiker and Zhang \cite{JMZ} has a combinatorial interpretation for the coefficients in terms of brane tilings. As graphs admissible on surfaces, they also have combinatorial interpretations in terms of snake graphs and paths on marked surfaces. Furthermore a paper by Glick and Weyman \cite{GW} has an explicit formula for the coefficients formed by summing a function over order ideals of a poset. An advantage of our formula over the ones listed is that it is more explicit; it does not abstractly summing over a combinatorial object we do not know how to count, which is a pitfall of all the above formulas. Ours only involves understanding the partition function, which is simple when there are only two possible parts the number can be divided into.

\section{Positivity of Exponents for Deformed $F$-Polynomials}
In this section, we begin our study of deformed $F$-polynomials, aspiring to prove that the deformed $F$-polynomial is, in fact, a polynomial.

\begin{lemma}
Take a quiver $Q$ and a sequence of mutations $v_1,\ldots,v_n$. Every monomial in any $F_n$ is expressible as the product of not necessarily distinct $r_i$ with $i\leq n$.
\end{lemma}

\begin{proof}
Assume there is a framed quiver $Q$ and sequence of mutations $v_1,v_2,\ldots$ such that some $F_n$ has a monomial not expressible as the product of the $r_i$'s. Take the smallest $n$ for which this is true, and the let the monomial of smallest degree in $F_n$ for which this is true be $m$. Let $V_{v_n}$ be the old label of vertex $v_n$. Let $S_1$ be the vertices corresponding to the set of base quiver vertices pointing in the same direction as the edges from frozen vertices. Let $S_2$ be the vertices corresponding to the other edges. Then, we have $$F_nV_{v_n}=r_n\prod_{i\in S_1} V_i + \prod_{i \in S_2} V_i$$ It is easy to see that each $F_i$ has a positive constant term of 1; namely $V_{v_n}$ does, so then $m$ is a term on the left hand side. Since there are no negative coefficients, the coefficient of $m$ on the left hand side must be positive. However, every term on the right hand side is expressible as a product of $r_i$'s since every term in every $F_i$ is, which is what the $V_i$'s are, and so is $r_n$.
\end{proof}

Recall our set $M_n$ for quiver $Q$ and mutations $v_1,\ldots,v_n$ which is the set of fundamental $r_i$ ($r_i$ not expressible as the product of other $r_j$).

\begin{lemma}
Every monomial in any $F_n$ is expressible as the product of not necessarily distinct elements of $M_n$. In fact, it's expressible as the product of not necessarily distinct elements of $M_m$ for any $m\ge n$, including $M$ itself.
\end{lemma}

\begin{proof}
By the previous lemma, we know every monomial is the product of $r_i$'s. Now, if an $r_i$ is not fundamental, express it as a product of other $r_i$'s. Repeating this process must result in a product of fundamental $r_i$'s, since degrees are decreasing.
\end{proof}

\begin{lemma}
Given a framed quiver $Q$ and a sequence of mutations $v_1,\ldots,v_n$, $m \in M_n$, its coefficient can be computed as follows: \\
Let $g$ be the number of $j\le n$ such that $m=r_j$ and $r_j$ is green, and let $r$ be the number of such $j\le n$ so that $r_j$ is red. Let $cf(m,V_i)$ be the coefficient of $m$ in $V_i$. Then $$-(g-r)C_n^{-1}(m)=\left<cf(m,V_1),cf(m,V_2),\ldots,cf(m,V_{v})\right>$$
\end{lemma}

\begin{proof}
First, we compute the coefficient of $m$ in $F_t$ for any $t$. If $m$ is fundamental, the only possible sequences of $r_i$'s producting to $m$ consists of the individual $r_i$'s themselves that equal $m$. For all such $r_i$, plugging into our formula for $F$-polynomials gives that the inner expression $\phi(\sf{w}) W(n, w_1,\ldots,w_k)$ simplifies to $-C^{-1}_{i,t}[v_n,v_i]$. Since $$-C^{-1}_{i,t}[v_n,v_i]=\delta_{0,i}C_t^{-1}(m)[v_t],$$ summing over all the $r_i$ equaling $m$ gives $-(g-r)C^{-1}_t(m)[v_t]$. Now, we compute $V_i$. Take $t$ so that $t$ is the last step less than or equal to $n$ where vertex $v_t$ is mutated. If no occurrences of $m$ as an $r_i$ are between $t+1$ and $n$ inclusive, then the coefficient of $m$ is in fact $-(g-r)C^{-1}_t(m)[v_t]$, where the $g$ and $r$ correspond to number of occurrences of $m$ preceding $n$. Note that $$C_n^{-1}=A_n\cdots A_{t+1} C_t^{-1}$$ where each mutation $A_i$ for $t+1\leq i \leq n$ corresponds to a vertex $v_i$. It is easy to see that left multiplication by this matrix only affects the $v_i$th row of the matrix. As such, no mutation affects row $v_t$, so $C^{-1}_t(m)[v_t]=C^{-1}_n(m)[v_t]$.

Now we have the second case, where some occurrence of $m$ proceeds spot $t$ in the mutation sequence. Take a such occurrence of this to be at index $s$. Then, by the way the $C$ matrix is defined, $C^{-1}_s(m)$ is the identity vector with a 1 at the spot $v_m \neq v_t$. Then, recalling that the $v_t$th row of the $C^{-1}$-matrix does not change between spot $t$ and $s$ and $m$, we have $$C^{-1}_t(m)[v_t]=C^{-1}_s(m)[v_t]=C^{-1}_n(m)[v_t]=0.$$ As such, the coefficient of $m$, which is $-(g-r)C^{-1}_t(m)[v_t]$ is 0, but this is also equal to $-(g-r)C_n^{-1}(m)[v_t]$ since this is also 0. Thus, in both cases, the expression simplifies to $-(g-r)C_n^{-1}(m)[v_t]$, and we are done.
\end{proof}

\begin{theorem} \label{posexps}
Take a framed quiver $Q$ and a mutation sequence $v_1,\ldots,v_n$, with the following property: for any fundamental $m \in M_n$, the number of occurrences of $i \leq n$ such that $r_i=m$ and $r_i$ is green is greater than the number where $r_i$ is red. Then, $S_n$ is a polynomial; that is none of the monomial terms have negative exponents if the $r_n$ satisfy this property.
\end{theorem}

\begin{proof}
Express any monomial in $S_n$ as a product of fundamental terms. Choose a sequence of fundamentals $m_1,\ldots,m_k$ and express as $\prod m_i$. Since matrix operations are linear, we have that $$-C^{-1}_n\left(\prod m_i\right)=\prod -C^{-1}_n(m_i).$$ We know that each monomial $-C^{-1}_n(m_i)$ has all positive exponents because of the following. Due to positivity of coefficients of $F_n$, every $m \in M_n$ is a valid monomial when acted on by $-C_n^{-1}$, since the coefficient vector multiplied by $(g-r)$ is the same as the exponent vector of $-C^{-1}_n(m_i)$, and the multiplication by $(g-r)$ preserves sign.
\end{proof}

\begin{corollary}
Given a framed quiver $Q$ and a mutation sequence $v_1,v_2,\ldots$ with every $r_i$ green, $S_n$ is a polynomial.
\end{corollary}

\begin{proof}
This is clear; for any $r_i$ the number of green occurrences must be at least the number of red occurrences since there are only green occurrences. Thus, this follows from Theorem \ref{posexps}.
\end{proof}

Almost conversely to Theorem \ref{posexps}, we have the following.

\begin{corollary} \label{rgfails}
Take a framed quiver $Q$ and a mutation sequence $v_1,\ldots,v_n$, with the following property: there exists a fundamental $m$ so that the number of occurrences of $i \le n$ such that $r_i=m$ and $r_i$ is red is greater than the number where $r_i$ is green. Then, $S_n$ is \textbf{not} a polynomial.
\end{corollary}

\begin{proof}
Consider the monomial corresponding to $m$: $-C_n^{-1}(m)$. We claim it will have all negative exponents. Its coefficient vector in the original cluster of $F$-polynomials is $-(g-r)C_n^{-1}(m)$. This has to be always positive due to positivity of coefficients. Since $r>g$, this gives us that $-C_n^{-1}(m)$ has all negative entries, as desired.
\end{proof}

Since we expect that the deformed $F$-polynomial is always a polynomial, we expect that the implied statement is always true.

\begin{conjecture}
Given a framed quiver $Q$ and a mutation sequence $v_1,v_2,\ldots$, it is true that for any $m \in M$, for any $n$ the number of occurrences of $i \le n$ such that $r_i=m$ and $r_i$ is green is at least the number where $r_i$ is red.
\end{conjecture}

The reason we cannot get an if and only if statement with this method directly is that the case where $g=r$ for a fundamental monomial $m$ is tricky to deal with. Replacing basic with fundamental provides an approach that circumvents this, though we do not know how to carry it through.

\begin{conjecture}\label{basiccon}
Given a framed quiver $Q$ and a sequence of mutations $v_1,\ldots,v_n$, $p \in P_n$, the coefficient of $p$ is the following: $$-C_n^{-1}(p)=\left< cf(p,V_1),cf(p,V_2),\ldots,cf(p,V_{v})\right>$$
\end{conjecture}

\begin{theorem}
If Conjecture \ref{basiccon} is true, then for any framed quiver $Q$ and mutation sequence $v_1,\ldots,v_n$, $S_n$ is a polynomial.
\end{theorem}

\begin{proof}
For any basic monomial, the coefficient vector is the same as the exponent vector, by the conjecture. Any monomial can be rewritten as the product of basics. Each of these primes has positive coefficient vector, and therefore positive exponent vector, so the product will have positive exponent vector.
\end{proof}

\begin{conjecture} \label{primecon}
Take a quiver $Q$ and a sequence of mutations $v_1,\ldots,v_n$. Let $\phi_{n,i}$ be the representation on the base quiver corresponding to the last step of the mutation process up to and including the $n$th step that occurred at vertex $i$. Let a monomial $p$ be prime if any subrepresentation of $\phi_{n,i}$ with dimensions corresponding to the degree vector of $p$ is indecomposable. Then the space of subrepresentations of $\phi_{n,i}$ with dimensions corresponding to the degree vector of $p$ has Euler-Poincare characteristic of the quiver Grassmannian equal to $-C^{-1}_n(p)_{i}$
\end{conjecture}

\begin{theorem}
If Conjecture \ref{primecon} is true, then for any framed quiver $Q$ and mutation sequence $v_1,\ldots,v_n$, $S_n$ is a polynomial.
\end{theorem}

\begin{proof}
For any prime monomial, the coefficient vector is the same as the exponent vector, by the conjecture. Any monomial can be rewritten as the product of primes; rewrite it as the direct sum of two representations, and inductively decompose it into primes. Each of these primes has positive coefficient vector, and therefore positive exponent vector, so the product will have positive exponent vector.
\end{proof}

\section{Convergence in Periodic Case for Green Quivers}

\begin{lemma}
Take a framed quiver $Q$ and a green mutation sequence $v_1,\ldots,v_n$. For any $r_i$, we have that $-C_{n}^{-1}(r_i)$ is a monomial with positive exponents.
\end{lemma}

\begin{proof}
We show each $r_i$ is the product of fundamentals. Assume $r_i$ is not a fundamental. Then, by definition, it is the product of other $r_j$'s. Repeating recursively, it is the product of fundamentals. Then, by a similar argument as before, $-C_{n}^{-1}(r_i)$ has to have all positive components, since $-C_{n}^{-1}$ of all its components in the decomposition into fundamentals is positive.
\end{proof}

\begin{lemma}
Given a green quiver $Q$ and a mutation sequence $v_1,v_2,\ldots$, all the $r_i$'s are distinct.
\end{lemma}

\begin{proof}
Assume monomial $m$ is both $r_i$ and $r_j$ for $j>i$. Consider the quiver after $j$ mutation steps. First we show $m$ is fundamental. If not, we will have $$-C^{-1}_j(m)= \prod-C^{-1}_j(r_a).$$ However, the left side is just the 0 vector with a 1 at $v_j$, and this has degree 1 so it cannot be the product of other terms. Now we know that the coefficient vector of a fundamental $r_i$ is simply $-(g-r)C^{-1}_j(r_i)$ which is 0 everywhere and 2 at $v_i$. In order for this to be true, at the previous step, the coefficient vector had to be 0 everywhere and $-1$ at $v_i$. This is because nothing other than $v_i$  changed; the new $r_i$ contributes 1, so the old term must contribute 1, and it was negated so must have been $-1$ before. This contradicts positivity of coefficients.
\end{proof}

\begin{lemma}
Take a green quiver $Q$, and a periodic mutation sequence $v_1,\ldots,v_p,\ldots,v_n$. For any monomial $m$ in the stabilized polynomial, there exists a constant $c$ so that for any sequence $w_1,\ldots,w_k$ with $(-C_n^{-1}(r_{w_1}))\cdots (-C_n^{-1}(r_{w_k})) = m$, then $n-c\leq w_1,\ldots,w_k$, independent of $n$.
\end{lemma}

\begin{proof}
Consider instead the monomials after acted on by $-C_n^{-1}$. Let the sequence of mutations be $1,\ldots,p,1,\ldots$, ending in $a$. We have $$-C_n^{-1}=A_n\cdots A_1 = A_a\cdots A_1A_p\cdots A_1 \cdots A_1.$$ For fixed $a$, consider each of the monomials $A_a\cdots A_1\cdots (e_{v_a})$, at each possible termination point of the matrix product, and assume two of them are equal. Say the length of the sequences at these two termination points are $b$ and $c$. Then take $k$ larger than both $b$ and $c$ and equal to $a \mod p$;. Then $$-C^{-1}_{k-b,k}(e_{v_a})=-C^{-1}_{k-c,k}(e_{v_a})$$ However then $$C(-C^{-1}_{k-b,k}(e_{v_a}))=C(-C^{-1}_{k-c,n}(e_{v_a})) \implies C_{k-b}(e_{v_{k-b}})=C_{k-c}(e_{v_{k-c}})$$ This means both $r_{v_{k-b}}$ and $r_{v_{k-c}}$ are equal. However, we already showed that all the $r_i$ are distinct in a green quiver. Thus, there is some $k$ for any monomial $p$ so that the only occurrence of $p$ has to be at mutation $n-k$ after stabilizing the corresponding $r_i$, assuming $n\equiv a \mod p$. Now, take the max of this value $k$ (call it $\max(k)$ for all monomials $p|m$ over all the congruence classes $\mod p$ (hence over all $n$)). Since only monomials $p|m$ can contribute to $m$, we get that all contributors must be after $n-\max(k)$, as desired. 
\end{proof}

\begin{theorem}
Given a green quiver $Q$, and a periodic mutation sequence $v_1,\ldots,v_p,\ldots$, for any $i$, $S_i,S_{i+p},\ldots$ converges as a formal power series.
\end{theorem}

\begin{proof}
Consider any monomial $m$ and we show its coefficient converges. Take $n$ larger than the $c$ value for monomial $m$ from the previous lemma. Thus, we have our equation: $$S_n = \sum_{0\leq w_1\ldots w_k\leq n-1} \phi(w_1,\ldots,w_k)W(n,n-w_k,\ldots,n-w_1) \prod_{i=1}^{v}y_i^{\sum_{j=1}^{k}\delta_{n-w_i,n}C^{-1}_{n-w_i,n}[j,v_{n-w_i}]}$$ which can be rewritten as $$S_n = \sum_{0\leq w_1\ldots w_k\leq c} \phi(w_1,\ldots,w_k)W(n,n-w_k,\ldots,n-w_1) \prod_{i=1}^{v}y_i^{\sum_{j=1}^{k}\delta_{n-w_j,n}C^{-1}_{n-w_j,n}[j,v_{n-w_j}]}$$ It would be clear these expressions were equal for sets of $n$ where we could show $C^{-1}_{n-w,n}$ was only dependent on $w$; i.e. for a fixed $w$, the same regardless of $n$. However, this is a property true of all $n$ in a congruence class $\mod p$. Thus, the expression is in fact equal for all $n > c$. 
\end{proof}

Recall from before that stability was originally conjectured by Eager and Franco in their paper \cite{EF}, using the $dP1$ quiver. As a green quiver, our paper proves this case. We also write an explicit formula for the deformed $F$-polynomial in the proceeding section.

\section{Exact Formulas for Deformed $F$-Polynomials}

In this section, we present some exact formulas for these limit deformed $F$-polynomials we have now shown to exist. We restrict to the green quiver case, since that's where we have shown convergence.

\subsection{Symmetric Quivers}
For this subsection, we write the exact deformed $F$-polynomials for symmetric quivers. It does not illustrate much, but allows for simpler proofs in proceeding sections.

\begin{theorem}\label{syms}
Let $S_n$ be the deformed $F$-polynomial for the symmetric quiver $Q$. Let $W$ be the set of (possibly empty) sequences $\sf{w}$ with $0\leq w_1 \leq \ldots \leq w_k\leq n-1$. Then $$S_n=\sum_{\sf{w}\in W} \phi(\sf{w}) \left( \prod_{i=1}^{k}s_{n-w_i}+\sum_{j=i}^{k}-s_{w_j-w_i}+s'_{w_j-w_i} \right)\prod_{i=1}^{v}y_{v_n+1-i}^{\sum_{j=1}^{k}s_{n-w_j+1-i}}.$$
\end{theorem}

\begin{proof}
For this proof, consider vertices by their value $\mod v$. Recall the formula from Theorem \ref{sym} which states $$F_n=\sum_{\sf{w}\in W} \phi(\sf{w}) \left( \prod_{i=1}^{k}s_{n-w_i}+\sum_{j=i}^{k}-s_{w_j-w_i}+s'_{w_j-w_i} \right)\prod_{i=1}^{v}y_i^{\sum_{j=1}^{k}s_{w_j+1-i}}.$$ We simply need to deform the exponents of the monomials. By Lemma \ref{symr}, $$r_k=\prod_{i=1}^{v}y_i^{s_{k+1-i}}$$ so the latter part of the expression can be rewritten as $$\prod_{j=1}^{k}r_{w_j}.$$ Now we deform the exponents to get $$\prod_{j=1}^{k}-C^{-1}_n(r_{w_j})$$ which is $$\prod_{j=1}^{k}C^{-1}_{w_j,n}(e_{v_{w_j}}).$$ Now since the quiver is symmetric, each $A_i$ is the same, with the vertices cyclic shifted according to $v_i$. Thus $C^{-1}_{w_j,n}$ will just be $A_{v_n}A_{v_n-1} \cdots A_{w_j}$ where there can be many periods in between. Mapping $v_n$ to vertex 1, $v_n-1$ to vertex 2 and so on, we get that $C^{-1}_{w_j,n}(e_{v_{w_j}})$ is $C_{0,n-w_j}(e_{v_n+1-v_{w_j}})$ under the new labeling. However, this is just $r_{n-w_j}$. Then we can use Lemma \ref{symr} to see that this is $$r_{n-w_j}=\prod_{i=1}^{v}y_i^{s_{n-w_j+1-i}}.$$ Returning the vertices to their original labels, we get that each original $r_k$ deforms to $\prod_{i=1}^{v}y_{v_n+1-i}^{s_{n-k+1-i}}$ and thus $$S_n=\sum_{\sf{w}\in W} \phi(\sf{w}) \left( \prod_{i=1}^{k}s_{n-w_i}+\sum_{j=i}^{k}-s_{w_j-w_i}+s'_{w_j-w_i} \right)\prod_{i=1}^{v}y_{v_n+1-i}^{\sum_{j=1}^{k}s_{n-w_j+1-i}}.$$
\end{proof}

\subsection{$\wt{A_{1,r}}$}

\begin{definition}
The quiver $\wt{A_{1,r}}$ is on vertices $1,\ldots,r+1$ with the following edges: 
\begin{itemize}
\item An edge from $i\to i+1$ for $1\leq i \leq r$.
\item An edge from $r+1 \to 1$.
\end{itemize}
\end{definition}

In this section, we look for explicit formulas for deformed $F$-polynomials for the quiver $\wt{A_{1,r}}$. In her paper \cite{Z}, Grace Zhang looks at the quiver $\wt{A_{1,1}}$. She shows the following. Let $S_{\infty}$ be the deformed $F$-polynomial for the quiver $\wt{A_{1,2}}$. Then, up to swapping $y_1$ and $y_2$, we have $$S_{\infty}=1+\frac{y_2}{(1-y_1y_2)^2}.$$

Our theorem generalizes this result.

\begin{theorem}
Let $S_{\infty}$ be the deformed $F$-polynomial for the quiver $\wt{A_{1,r}}$. Then $$S_{\infty}=1+\frac{(y_{r+1}+y_{r+1}y_{r}+\ldots+y_{r+1}\cdots y_2)}{(1-y_1\cdots y_{r+1})^2}$$ up to a cyclic shift of variables.
\end{theorem}

\begin{proof}
Let $W$ be the set of sequences $0\leq w_1\leq \ldots \leq n-1$. By Theorem \ref{syms} looking at indices modulo $r$ we know that $$S_n=\sum_{\sf{w}\in W} \phi(\sf{w}) \left( \prod_{i=1}^{k}s_{n-w_i}+\sum_{j=i}^{k}-s_{w_j-w_i}+s'_{w_j-w_i} \right)\prod_{i=1}^{v}y_{v_n+1-i}^{\sum_{j=1}^{k}s_{n-w_j+1-i}}$$ where $s_i=\lc \frac{i}{r} \rc$. 

Consider any monomial $m$ that we are trying to find the coefficient of in the stabilized polynomial which has nonzero coefficient. We claim that the exponent of $y_{v_n}$ is at most one greater than the exponent of $y_{v_n+1}$ in any term $m$ with nonzero coefficient. Because $\sf{s}$ increases every $r$ terms and there are $r+1$ vertices, the difference between the exponent of $y_{v_n}$ and $y_{v_n+1}$ is going to be $k$. This also means that all sequences creating a given monomial are all the same length. Thus, we just have to show that the only terms with nonzero coefficient in the limit are the ones with all corresponding sequences $\sf{w}$ having length 1. Assume some sequence $\sf{w}$ with length greater than 1 corresponds to some monomial which has nonzero coefficient in the limit, specifically for $S_n$ and all future $S_{n+k(r+1)}$. In $S_n$ it corresponds to $w_1,\ldots,w_k$. Then consider the polynomial $S_{n+n(r+1)}$. The sequences $w_1+n(r+1),\ldots,w_k+n(r+1)$ would contribute to this monomial in $S_{n+n(r+1)}$. However, in the original un-deformed polynomial, this would mean the corresponding sequence $r_{w_1+n(r+1)}\cdots r_{w_k+n(r+1)}$ products to the corresponding term. However, this terms clearly has degree too large in the original polynomial, so must have 0 coefficient. As such, any term of that form will eventually drop out of the limit. Therefore, it suffices to sum over all sequences $\sf{w}$ with length 0 or 1. The sequences with length 0 only create the term 1. Letting $S'_{n}$ be the deformed polynomial summing over only length 0 and 1 sequences, we get $$S'_{n}=1+\sum_{1\leq w \leq n} s_{n-w}\prod_{i=1}^{v}y_{v_n+1-i}^{s_{n-w+1-i}}$$ where $s_i=\lc \frac{i}{r} \rc$. Since we only care about the $y_i$ up to cyclic shift, and instead summing over the length 1 sequences $\sf{q}$ with $1 \leq q_1 \leq n$ mapping $\sf{w}$ to $n-\sf{q}$, we can rewrite the above as $$S'_{n}=\sum_{1\leq q \leq n} s_{q}\prod_{i=1}^{v}y_{v+1-i}^{s_{q+1-i}}$$ up to cyclic shift of variables, and taking the limit as $n\to \infty$ gives $$S'_{\infty}=\sum_{1\leq q} s_{q}\prod_{i=1}^{v}y_{v+1-i}^{s_{q+1-i}}.$$ A sequence of exponents occurs if and only if the exponent of $y_v$ is one more than the exponent of $y_1$ and the exponent sequence is nonstrictly decreasing. Thus, we get $$S'_{\infty}=\sum_{0\leq e} (e+1) \left(y_{r+1}^{e+1}y_r^e\cdots y_1^e + y_{r+1}^{e+1} + y_r^{e+1}\cdots y_r^e +\ldots + y_{r+1}^{e+1} \cdots y_2^{e+1}y_1^e \right).$$ Because the limits of the $S'$ and the $S$ are the same, this simplifies to $$S_{\infty}=1+\frac{y_{r+1}(1+y_{r}+y_{r}y_{r-1}+\ldots+y_{r}\cdots y_2)}{(1-y_1\cdots y_{r+1})^2}$$ as desired.
\end{proof}

\subsection{$K_r$}

In this section, we investigate the deformed $F$-polynomial for the $K_r$ quiver. It is worthwhile to note that $K_2=\wt{A_{1,1}}$. As such, this case can be viewed as a different generalization of the Zhang's result for the $\wt{A_{1,1}}$ quiver. We redefine the relevant sequences.

\begin{definition}
Define the sequence $s$ as $\ldots,s_{-2}=s_{-1}=0, s_0=1$ and $s_{k+2}=rs_{k+1}-s_k$ for $k\ge 0$. Let $p=\frac{r+\sqrt{r^2-4}}{2}$ Let the norm at $b$ of a sequence $1\leq w_1 \leq \ldots\leq w_k$ be $$N_b(\sf{w})=\sum_{i=1}^{k} b^{w_i}.$$ The explicit value of any given $s_i$ for $i\geq 0$ can be calculated as follows: $$s_i = \frac{1}{\sqrt{r^2-4}}p^i - \frac{1}{\sqrt{r^2-4}}p^{-i}.$$
\end{definition}

\begin{theorem}
The deformed $F$-polynomial $S_{\infty}$ is calculated as follows. Let $Q$ be the set of (possibly empty) sequences $\sf{w}$ with $0\leq w_1\leq \ldots\leq w_k$ with $N_{1/p}(\sf{w})$ at most 1. $$S_{\infty}=\sum_{\sf{w}\in W} \phi(\sf{w}) \prod_{i=1}^{k} \left( s_{w_i} - \sum_{j=1}^{i} s_{w_i-w_j} + s_{w_i-w_j-2} \right)y_1^{\sum_{i=1}^{k}s_{w_i}}y_1^{\sum_{i=1}^{k}s_{w_i-1}}$$ Alternately, $$S_{\infty}=\sum_{\sf{w}\in W} \phi(\sf{w}) \prod_{i=1}^{k} \left( s_{w_i} - \sum_{j=1}^{i} s_{w_i-w_j} + s_{w_i-w_j-2} \right)y_1^{\lf \frac{1}{\sqrt{r^2-4}}pN_p(q)\rf}y_2^{\lf \frac{1}{\sqrt{r^2-4}}N_p(\sf{w})\rf}$$ up to a permutation of variables.
\end{theorem}

\begin{proof}
Recall the statement of our formula from Theorem \ref{syms}: $$S_n=\sum_{\sf{w}\in W} \phi(w_1,\ldots,w_k) \left( \prod_{i=1}^{k}s_{n-w_i}+\sum_{j=i}^{k}-s_{w_j-w_i}+s'_{w_j-w_i} \right)\prod_{i=1}^{v}y_{v_n+1-i}^{\sum_{j=1}^{k}s_{n-w_j+1-i}}.$$ Noting that there are only two vertices, this simplifies to $$S_n = \sum_{1\leq w_1\ldots w_k\leq n} \phi(w_1,\ldots,w_k) \prod_{i=1}^{k} \left( (s_{w_i}) - \sum_{j=1}^{i-1} (s_{w_j-w_i}+s_{w_j-w_i-2}) \right) y_1^{\sum_{j=1}^{k}s_{w_i}}y_2^{\sum_{j=1}^{k}s_{w_i-1}}$$ where $y_1$ and $y_2$ are swapped if $n$ is even. We ignore this distinction in the proof that follows. This is exactly our desired expression, except we need to limit the set of sequences we are summing over, using the norm condition. We take a monomial $m$ with degree in $y_1$ being $d$. Then, all possible sequences $w$ that allow for this monomial have at most $d$ terms and all $w_i<d$ ($d$ is a very crude bound, but suffices). Take $n$ extremely large. Now, we look at the degree of the term corresponding to $w$ in the original $F$-polynomial. We want to show both that if the norm is at most $1$, the degree will be this must be less than or equal to $s_{n-1}$, and if norm is greater than 1 it will be larger. The degree is $$\sum_{i=1}^{k}s_{n-w_i} = \frac{1}{\sqrt{r^2-4}}(p^nN_{1/p}(w) - \sum_{i=1}(1/p)^{n-w_i})$$ by using the explicit form of the $s_i$. $s_{n-1}$ can be expressed as $$\frac{1}{\sqrt{r^2-4}}p^n- \frac{1}{\sqrt{r^2-4}}p^{n}$$ When $n$ goes to infinity, the first term of each expression dominates, and this easily proves the desired result; the first is smaller when norm is less than 1, and larger otherwise. If the norm is exactly 1, the first terms of each expression are equal, so we compare the second term. The first expression clearly has a smaller first term, so the result holds in this case as well.

Now we prove the second part; that $$\sum_{j=1}^{k}s_{w_i}=\lf \frac{1}{\sqrt{r^2-4}}N_p(q)\rf$$ and $$\sum_{j=1}^{k}s_{w_i-1}=\lf \frac{1}{\sqrt{r^2-4}}pN_p(q)\rf$$ for a sequence with norm at most 1. For the first part, first rewrite $$\sum_{j=1}^{k}s_{w_i}=\frac{1}{\sqrt{r^2-4}}(pN_p(w)-p^{-1}N_{1/p}(w))$$ The second term (after expanding) is strictly greater than 0 and less than 1, and since the expression must be an integer, we have proven our desired result. We have $$\sum_{j=1}^{k}s_{w_i}=\frac{1}{p\sqrt{r^2-4}}(N_p(w)-N_{1/p}(w))$$ Again the second term after expanding has to be greater than 0 and less than 1, proving the desired result.
\end{proof}

We now find another formulation for the deformed $F$-polynomial of the quiver $K_r$ highlighting exactly which monomials have nonzero coefficient.

\begin{corollary}
The coefficient of the term $y_1^{\lf pa\rf +1}y_2^{a}$ for integers $a\ge 0$ are the only ones that are nonzero, and are expressed as follows. Let $W$ be the set of sequences $\sf{w}$ with $0\leq w_1 \leq \ldots\leq w_k$ and with $\sum_{\sf{w}\in W}s_{w_i-1}=a$. Then, the coefficient corresponding to that value of $a$ is $$\phi(w_1,\ldots,w_k) \prod_{i=1}^{k} \left( (s_{w_i}) - \sum_{j=1}^{i-1} (s_{w_j-w_i}+s_{w_j-w_i-2}) \right)$$
\end{corollary}

\begin{proof}
We need to find $$\lf p\lf \frac{1}{\sqrt{r^2-4}}N_p(q)\rf\rf +1$$ Recall from the previous proof that our expression of $\lf \frac{1}{\sqrt{r^2-4}}N_p(q)\rf$ is also $$\frac{1}{\sqrt{r^2-4}}N_p(q)-\frac{1}{\sqrt{r^2-4}}N_{1/p}(q)$$ Multiplying by $p$ gives $$\frac{p}{\sqrt{r^2-4}}N_p(q)+\frac{p}{\sqrt{r^2-4}}N_{1/p}(q)$$ which is $$\sum_{i=1}^{|q|}s_{w_i}+\frac{1}{p\sqrt{r^2-4}}N_{1/p}(\sf{w})-\frac{1}{\sqrt{r^2-4}}N_{1/p}(\sf{w})$$ which is $$\sum_{i=1}^{|q|}s_{w_i}+\frac{p-1}{p}\frac{1}{\sqrt{r^2-4}}N_{1/p}(q)$$ The second term is non-strictly less than 1, but greater than 0, so taking the floor and adding 1 gives $\sum_{i=1}^{|q|}s_{w_i}$ as desired for the first part.

For the second part, we know that we get a given term by summing the expression $$\phi(w_1,\ldots,w_k) \prod_{i=1}^{k} \left( (s_{w_i}) - \sum_{j=1}^{i-1} (s_{w_j-w_i}+s_{w_j-w_i-2}) \right)$$ over sequences $\sf{w}$ with $0\leq w_1 \leq \ldots\leq w_k$ with $\sum_{\sf{w}\in W}s_{w_i-1}=a$. Thus, it suffices to prove that $a=\lf \frac{p}{\sqrt{r^2-4}}N_p(q) \rf$. Note that $$s_{w_i-1}=\frac{1}{p\sqrt{r^2-4}}N_p(\sf{w})-\frac{p}{\sqrt{r^2-4}}N_{1/p}(\sf{w}).$$ The second term is less than 1, so we get that the overall expression is just the floor of the first part, as desired.
\end{proof}

The deformed $F$-polynomial is interesting for two reasons: it allows for the convenient norm property to determine which monomials are included. Further, the renormalization has $y_1$ and $y_2$ uniquely determine each other.

\subsection{Gale Robinson Quivers}

We now look at the explicit deformed $F$-polynomial for the Gale Robinson Quivers. Essentially we just plug into our explicit formula and nothing extremely interesting happens, but it allows us to completely close the Eager and Franco's original conjecture data around the $dP1$ quiver.

\begin{theorem}
Take the quiver $G_{v,r,t}$. Define a sequence $\sf{s}$ with $s_i$ as the number of partitions of $i$ into parts $r,v-r$ (0 for negative values of $i$, 1 for 0). Let $W$ be the set of sequences $\sf{w}$ with $0\leq w_1 \leq \ldots \leq w_k$.
$$S_{\infty}=\sum_{\sf{w} \in W} \phi(w_1,\ldots,w_k)\prod_{i=1}^{k} \left( s_{w_i} + \sum_{j=1}^{i-1} (-s_{w_j-w_i}-s_{w_j-w_i-v}+s_{w_j-w_i-t}+s_{q_j-q_i-n+t}) \right) \prod_{i=1}^{v}y_i^{\sum_{j=1}^{k}s_{w_i-j+1}}$$
\end{theorem}

\begin{proof}
We have by Theorem \ref{syms} that
$$S_n = \sum_{0\leq w_1\ldots w_k\leq n-1} \phi(w_1,\ldots,w_k)\prod_{i=1}^{k} \left( s_{w_i} + \sum_{j=1}^{i-1} (-s_{w_j-w_i}-s_{w_j-w_i-v}+s_{w_j-w_i-t}+s_{w_j-q_i-n+t}) \right) \prod_{i=1}^{v}y_i^{\sum_{j=1}^{k}s_{w_i-j+1}}$$
since that is how to calculate the corresponding sequences, and so taking the limit gives
$$S_{\infty}=\sum_{0\leq w_1 \ldots w_k} \phi(w_1,\ldots,w_k)\prod_{i=1}^{k} \left( s_{w_i} + \sum_{j=1}^{i-1} (-s_{w_j-w_i}-s_{w_j-w_i-v}+s_{w_j-w_i-t}+s_{w_j-w_i-n+t}) \right) \prod_{i=1}^{v}y_i^{\sum_{j=1}^{k}s_{w_i-j+1}}$$
\end{proof}

Now we find the exact deformed $F$-polynomial for the quiver $dP1=G_{4,2,1}$. If we can, it is always more interesting to calculate the coefficient of a given monomial directly, rather than summing up a number of terms that could possibly be contributing to the same polynomial. We will do this for the dP1 quiver, since the sequence $\sf{s}$ is easy to understand.

\begin{corollary}
Take the quiver $dP1$ and mutation sequence $1,2,3,4,1,2,\ldots$ which has $n$ terms. The coefficient of the term $y_1^ay_2^by_3^cy_4^d$ is calculated as follows. Let $W$ be the set of sequences of nonnegative integers and half-integers so that
\begin{itemize}
\item $a-c$ terms are integers.
\item $b-d$ terms are half-integers.
\item The sum of all the integer terms is $c$.
\item The sum of the floors of all the half integer terms is $d$.
\end{itemize}
Then the desired coefficient is $$\sum_{\sf{w}\in W}\phi(w_1,\ldots,w_k)\prod_{i=1}^{k} \left( \frac{(-1)^{2w_i}+1}{2}(w_i+1) - \sum_{j=1}^{i-1} (-1)^{2w_j-2w_i}(2w_j-2w_i + 4) \right)$$
\end{corollary}

\begin{proof}
Let $s(k)$ be $k$ if $k$ is a positive integer, and 0 otherwise. Noting that $s_{k}$ as defined in the Gale Robinson quiver is $\frac{k}{2}$ if $k$ is even and 0 otherwise, plugging in gives:
$$S_{\infty}=\sum_{0\leq w_1 \ldots w_k, w_i} \phi(w_1,\ldots,w_k)\prod_{i=1}^{k} \left( \frac{((-1)^{w_i}+1)w_i}{4} - \sum_{j=1}^{i-1} (-1)^{w_j-w_i}(w_j-w_i + 4) \right) \prod_{i=1}^{4}y_i^{\sum_{j=1}^{k}\frac{(-1)^{w_i-j+1}+1}{4}}$$ 
Letting $q_i=\frac{w_i}{2}$ gives the following expression. $$S_{\infty}=\sum_{q\in Q} \phi(q_1,\ldots,q_k)\prod_{i=1}^{k} \left( q(w_i) - \sum_{j=1}^{i-1} (-1)^{2q_j-2q_i}(2q_j-2q_i + 4) \right) \prod_{i=1}^{4}y_i^{\sum_{j=1}^{k}s(q_j-\frac{1}{2}i)}$$ For any term $y_1^ay_2^by_3^cy_4^d$ it is contributed to by the sequences described in the corollary statement. Adding the contributions proves the desired result.
\end{proof}

\section{Future Research}

\subsection{The Formula}
One important avenue of future research involving our formula in Theorem \ref{formula} consists of finding more use cases for the formula. For example, it might be possible to group terms and use the formula to recover positivity of coefficients, at least in certain cases. We might be able to do this by grouping terms of the formula. 

Furthermore, we know that the Laurent Phenomenon proves that the $F$-polynomials are in fact polynomials. However, our formula has nontrivial expansion for the coefficient of all monomials. Thus, for monomials that are too large, or obviously otherwise have coefficient 0, we get an expression involving the $C$-matrix entries that must be zero. We have not yet investigated these; it might be interesting. It might even help with positivity; one might be able to say that if large terms have coefficient 0, the smaller terms must have coefficient positive. The following provides intuition for why such a proof might work.

\begin{lemma}\label{posex}
Take a quiver $Q$ and a mutation sequence $v_1,\ldots,v_n$. If each $r_i^k$ can only be constructed as a product of $r_i$'s, and not any of the other terms, then the coefficient of $r_i$ must be positive.
\end{lemma}

\begin{proof}
The coefficient of $r_i$ by our formula is $a_{i,n}$, and the coefficient of $r_i^k$ is $$\frac{1}{k!}\prod_{i=1}^{k}\left(a_{i,n}-\sum_{j=1}^{k-1}(a_{i,i}-b_{i,i})\right)$$ which is $$\frac{1}{k!}\prod_{i=1}^{k}( a_{i,n}-k+1)$$ This is only 0 for large $k$ when $a_{i,n}$ is positive.
\end{proof}

This idea is complicated to generalize, since in large enough monomials, every part-contribution (the particular sequence we are summing over) need not be 0. However, there might be a clever way to group terms to deal with this.

Another thing we want to do with our formula is re-express it without the $\phi(w_1,\ldots,w_k)$ term. Our formula currently reads 
$$F_n = \sum_{\sf{w}\in W} \phi(w_1,\ldots,w_k) W(n, w_1,\ldots,w_k) \prod_{i=1}^{k}r_{w_i}$$ over all sequences $\sf{w}$ with $1\leq w_1 \leq \ldots \leq w_k\leq n$. This can be rewritten as $$F_n = \sum_{\sf{q}\in Q} \frac{1}{k!} W(n, \sigma(q_1,\ldots,q_k)) \prod_{i=1}^{k}r_{q_i}$$ over all sequences $\sf{q}$ with $1\leq q_1 ,\ldots, q_k\leq n$ (not increasing order), where $\sigma(q_1,\ldots,q_k)$ is the sequence permuted to be in increasing order. This is because the extra occurrences of each sequence $\sf{w}$ exactly account for the $\phi$ term. However, this formula isn't any more insightful than the the original one, because the $W$ term still involves permuting the sequence to be in increasing order. Our goal would be to find an explicit expression $W'(n, q_1,\ldots,q_k)$ so that $$F_n = \sum_{\sf{q}\in Q} \frac{1}{k!} W'(n, q_1,\ldots,q_k) \prod_{i=1}^{k}r_{q_i}.$$ In particular, we want an explicit formula for $W'(n, q_1,\ldots,q_k)=W(n, \sigma(q_1,\ldots,q_k))$ that does not involve first permuting the sequence $\sf{q}$ into increasing order.

Another possible avenue of research would be to see how our work relates to dilogarithm identities. Keller \cite{K2} and Nakanashi \cite{N} discuss these identities, and they appear to be an area where explicit formulas may allow for better results.

\subsection{Deformed $F$-Polynomials}

Regarding deformed $F$-polynomials, or $S$-polynomials we have the following three conjectures.

\begin{conjecture}
Given a framed quiver $Q$ and a mutation sequence $v_1,v_2,\ldots$, $S_n$ is actually a polynomial; that is none of the monomial terms have negative exponents.
\end{conjecture}

\begin{conjecture}
Given a framed quiver $Q$ and a mutation sequence $v_1,v_2,\ldots$, $S_n$ is bounded; i.e. for each monomial term $\prod x_i^{a_i}$, there is some $c(a_1,\ldots,a_n)$ dependent on $Q$ and the mutation sequence so that the coefficient of $\prod x_i^{a_i}$ in $S_n$ is less than $c(a_1,\ldots,a_n)$.
\end{conjecture}

\begin{conjecture}
Given a framed quiver $Q$ and a periodic mutation sequence $v_1,\ldots,v_p,v_1,\ldots,v_p,\ldots$, such that performing the mutations $v_1,\ldots,v_p$ on the base quiver returns the quiver back to the original, for every $i$, the sequence $S_i, S_{i+p}, S_{i+2p},\ldots$ converges. 
\end{conjecture}

We proved the first and third in the case of green quivers. The first is very related to the coefficients in our formula being positive. Recall that our proof of positivity in the green case relied on positivity of coefficients. It might be possible to generalize this, using our formula to understand the coefficient of ``basic'' terms via the formula. This would in turn prove Conjecture \ref{basiccon}, thus proving the result. We might also be able to achieve a direct proof using a concept like in Lemma \ref{posex}.

We have not spent much time investigating the second conjecture. We need to extend the proof used for convergence in the periodic green case in two ways. We need to remove the condition periodic and the condition green. It would be especially interesting and probably doable to remove the periodic condition.

For the third conjecture, again, we have proven it in the case of green quivers. The arguments do not extend directly from the green case, though a proof is likely similar to a full proof of the previous conjectures.

\section{Acknowledgments}
This research was carried out as part of the 2018 Mathematics REU at University of Minnesota, Twin Cities. It was supported by the NSF RTG grant DMS-1745638. The author would like to thank mentor Gregg Musiker for his invaluable guidance and support, numerous read throughs and edits of the paper, and mathematical insight. The author would also like to thank TA Esther Banaian for her comments, as well as researchers Kyungyong Lee and Joe Buhler for reading and offering edits to the paper. Much of this research was also aided by the open source mathematical software Sage.

\medskip

\bibliographystyle{plain}
\bibliography{fpoly}

\end{document}